\documentclass[11pt]{article}
\usepackage{amsthm}
\usepackage{amsmath}
\usepackage{amssymb}
\usepackage[margin=1in]{geometry}
\usepackage{graphicx}
\usepackage{algorithm2e}
\usepackage{subcaption}
\usepackage[numbers]{natbib}
\usepackage[backref=page]{hyperref}
\hypersetup{
    colorlinks=true,     
    linkcolor=blue,      
    citecolor=blue,      
    filecolor=blue,      
    urlcolor=blue        
}

\newtheorem{proposition}{Proposition}
\newtheorem{thm}{Theorem}
\newtheorem{lemma}{Lemma}
\newtheorem{definition}{Definition}

\newtheorem{example}{Example}
\newtheorem{remark}{Remark}

\newtheorem{assumption}{Assumption}

\DeclareMathOperator*{\argmin}{\arg\!\min}
\DeclareMathOperator*{\argmax}{\arg\!\max}

\newcommand{\tr}{{\mathbf{tr}}}

\providecommand{\inner}[2]{\langle {#1} , {#2} \rangle}		   

\begin{document}

\title{Competitive Online Algorithms for Resource Allocation over the Positive Semidefinite Cone}

\author{Reza Eghbali\thanks{Tetration Analytics, Cisco Systems, Inc. Palo Alto, CA 94301 
	\texttt{reghbali@cisco.com, eghbali@uw.edu}}
	\and James Saunderson\thanks{Department of Electrical and Computer Systems Engineering, Monash University, Victoria 3800, Australia, \texttt{james.saunderson@monash.edu}}
	\and Maryam Fazel\thanks{Department of Electrical Engineering, University of Washington, Seattle, WA 98195
	\texttt{mfazel@uw.edu}}}

\maketitle

\begin{abstract}
We consider a new and general online resource allocation problem,  where the
goal is to maximize a function of a positive semidefinite (PSD) matrix with a
scalar budget constraint. The problem data arrives online, and the algorithm
needs to make an irrevocable decision at each step. Of particular interest are
classic experiment design problems in the online setting, with the algorithm
deciding whether to allocate budget to each experiment as new experiments
become available sequentially.

We analyze two greedy primal-dual algorithms and provide bounds on their
competitive ratios.  Our analysis relies on a smooth surrogate of the objective
function that needs to satisfy a new diminishing returns (PSD-DR) property
(that its gradient is order-reversing with respect to the PSD cone).  Using the
representation for monotone maps on the PSD cone given by L\"owner's theorem,
we obtain a convex parametrization of the family of functions satisfying
PSD-DR. We then formulate a convex optimization problem to directly optimize
our competitive ratio bound over this set. This design problem can be solved
\emph{offline} before the data start  arriving. The online algorithm that uses
the designed smoothing is tailored to the given cost function, and enjoys a
competitive ratio at least as good as our optimized bound.  We provide examples
of computing the smooth surrogate for D-optimal and A-optimal experiment
design, and demonstrate the performance of the custom-designed algorithm.
\end{abstract}

\section{Introduction} \label{sec:intro}

Online resource allocation problems and algorithms have been traditionally
studied in operations research (e.g., revenue management
\cite{agrawal2009dynamic} and references therein), online network routing
\cite{buchbinder2006improved} and computer science (e.g., online packing and
covering in linear or convex case \cite{buchbinder2009design,azar2016online},
online welfare maximization \cite{kapralov2013online}).  In recent years,
applications in online advertising such as the Adwords problem
\cite{mehta2007adwords} have garnered renewed interest in this topic.  In all
these applications, the demands for resources arrive in an online, sequential
fashion, and resource allocation also happens online. For example, in online
revenue management, customers arrive sequentially, reveal their demand for the
resources and offer a bid price. The inventory owner then needs to make an
irrevocable decision about the latest customer's offer, without knowing future
demands, while aiming to maximize his or her revenue.  A feature in these
problems is that the variables denoting allocation amounts, the bid prices, and
the coefficients in the resource constraints are all \emph{nonnegative}, thus a
resource can only be used up as time goes on (resource usage is nondecreasing)
while the total available resource is fixed.  In the context of linear
programming (LP), these are sometimes called ``packing" problems, e.g.,
\cite{molinaro2013geometry,buchbinder2009design}.  For an overview of online
linear programming and more generally online convex problems with different
models for online information arrival, we refer the reader
to~\cite{legrain2016stochastic,gupta2016experts,agrawal2015fast,eghbali2014exponentiated}. 

In this paper, we consider a new and general online resource allocation
problem, where the objective is a function of a positive semidefinite (PSD)
matrix (thus the problem is defined on the PSD cone $S_{+}^n$), with a scalar
budget constraint.  The problem can be described as follows: At round $t$, the
algorithm receives a matrix $A_t \in S_{+}^n$ and scalar $c_t > 0$, and then
chooses $x_t \in [0,1]$. The goal of the algorithm is to maximize
$H\left(\sum_{t=1}^{m} A_t x_t\right)$ subject to the budget constraint
$\sum_{t=1}^{m} c_t x_t \leq b$. The offline problem can be written as
\begin{equation}\label{main}
\mbox{maximize}\quad H\left( \sum_{t=1}^{m} A_t x_t\right) \qquad\mbox{subject to}\quad
	\begin{cases} 
		\sum_{t=1}^{m}c_t x_t \leq b,&\\
		0 \leq x_t \leq 1, \;\; t=1,\ldots,m.
	\end{cases}
\end{equation}
We assume that $H: S_{+}^n \rightarrow \mathbf{R}$ is a monotone, concave \emph{trace function}, i.e.,
\vspace{-.01in}
\[
H\left(X\right) = {\textstyle \sum_{i=1}^{n} } h\left(\lambda_i\left(X\right)\right),
\]
where $h: \mathbf{R}_{+} \rightarrow \mathbf{R}$ is a monotone concave function
and $\lambda_i(X)$ denotes the $i$th eigenvalue of $X$. Note that analogous to
packing LPs, $\sum_{s=1}^t  A_s x_s$ is nondecreasing with respect to the PSD
cone. In other words, 
\[ \sum_{s=1}^{t}A_sx_s \preceq \sum_{s=1}^{t+1}A_sx_s\quad\textup{for all $t$,}\]
where $X \preceq Y$ means that $Y-X$ is positive semidefinite. 

In the online algorithms literature, various scenarios have been considered for
the information being revealed to the algorithm (the sequence of $A_t$ and $c_t$): from the worst-case
adversarial model to various stochastic models. In this paper, since we are interested in
understanding the limits of performance and providing competitive ratio bounds,
we assume an adversarial model, where little is known about the arriving data.
One could use techniques similar to ours in other scenarios as well, which we
leave for future work. 

Our setup covers online versions of problems such as \emph{sensor selection} or
\emph{experiment design}, and \emph{graph formation} \cite{cdc2016formation}. 
In these problems, $A_t = a_t a_t^T $ for some $a_t \in \mathbf{R}^n$. In online
experiment design, the vector $a_t$ is an experiment or measurement vector that
provides linear noisy measurements of an unknown vector $w$, i.e.,
$\xi_t=\inner{a_t}{w} + n_t$, where $n_t$ is Gaussian noise and $c_tx_t$ is the
experiment cost. In this paper, we consider two algorithms: one that assigns
fractional values to $x_t$ and one that assign integer values to $x_t$. 
The algorithm makes online decisions about the budget to allocate to an
experiment, aiming to minimize various functions of the error covariance matrix
of the maximum a priori estimate of $w$. 

The online experiment design problem can be expressed as \eqref{main}, with different choices of $H$
giving rise to different criteria for optimal experiment design. Examples
include 
$H \left(X \right) = \log\det\left( \epsilon I+X \right)$ for the D-optimal
criterion, $H\left(X\right) = - \tr{\left(\epsilon I+X\right)^{-1}}$  for the
A-optimal criterion, and $H\left(X\right) = - \tr{\left(\epsilon I+X\right)^{-
p}}$ for the $p$th mean criterion. (The prior distribution on $w$ is assumed to
be $\mathcal{N}\left(0, \frac{1}{\epsilon} I \right)$.) 

We use the A-optimal
and D-optimal  criteria, as well as the problem where 
$H(X) = \tr{\left(X\right)}$, as running examples throughout the paper. 
The latter simply recovers a linear program with one packing constraint, and
provides insight into the connections and differences with online Linear Programming results, e.g., \cite{buchbinder2009design} (discussed in section \ref{sec:compratio}).

Our results exploit a crucial property of the objective function $H$ or its
surrogate $H_S$ (discussed in section \ref{sec:smooth-h}) which we refer to as
the \emph{PSD diminishing returns} property. For a function with a scalar
variable, this property simply means the larger the variable, the smaller the
derivative. More generally, we define it as follows. 
\begin{definition}[PSD Diminishing Returns]\label{assumption1} 
	A concave trace function $F$ satisfies \emph{PSD diminishing returns (PSD-DR)} if 
	$\nabla F$ is order-reversing with respect to $S_+^n$, i.e., 
	$$\nabla F\left(U\right) \preceq \nabla F\left(U'\right), \text{  whenever }  U \succeq U' .$$
\end{definition}

\paragraph{Our contributions.} 
Our focus is on developing algorithms with (multiplicative) competitive ratio
guarantees for problems of the form~\eqref{main}. In other words, our
algorithms provide feasible solutions $x_1,x_2,\ldots,x_m$ such that $
H\left(\textstyle{\sum_{t=1}^{m}} A_t x_t\right) - H\left(0\right) \geq
\mbox{cr} \left(P^\star - H\left(0\right)\right)$, where $\textup{cr}$ bounds
the competitive ratio and $P^\star$ is the optimal value
of~\eqref{main}\footnote{To simplify the notation in the rest of the paper,
assume $H(0) = 0$ by replacing $h(u)$ with $h(u) - h(0)$}.  

The algorithms we consider (Algorithms~\ref{algorithm1} and~\ref{algorithm2},
in Section~\ref{sec:alg}) are parameterized by a smoothed surrogate, $H_S$, of
the objective function $H$, and a smoothed penalty function $G_S$ in place of
the budget constraint $\sum_{t=1}^{m}x_t \leq b$.  Our main contribution is a
systematic method, based on solving an offline convex optimization problem, to
design $H_S$ and $G_S$, and to compute a bound on the competitive ratio of the
associated algorithms. 

Our approach builds on the results in \cite{eghbali2016smoothing}, which
considers a general online conic optimization problem and shows that smoothing
the objective function (to obtain a surrogate) before applying a primal-dual
greedy algorithm improves the competitive ratio.  However, prior work was not
able to address the problem of finding the best surrogate (or even representing
the space of suitable surrogates)---except for the special (and restrictive)
case of separable objectives defined on the non-negative orthant, which reduces
to a search over concave scalar functions.
There is no hope of employing a similar strategy for functions defined on the
PSD cone, unless one has access to  representations that lead to a convex
parametrization for PSD-DR functions. The present paper resolves this issue. 

We require the smoothed objective $H_S$ to satisfy the PSD-DR
property. Note that this does not hold for all concave trace
functions (indeed fails to hold for the objective function in A-optimal
experiment design). If the objective function in problem \eqref{main} does not
satisfy this property, our approach allows us to design an appropriate
surrogate $H_S$ that satisfies it, and use this $H_S$ to construct algorithms
with competitive ratio guarantees.  And if the original $H$ does satisfy the
property, our approach still helps improve the competitive ratio; see section
\ref{sec:smooth-h}.  A key observation in this paper is that we can use L\"owner's
theorem, characterizing operator monotone functions, to impose this constraint
in a computationally effective way by requiring it to have a certain integral
representation (see~\eqref{eq:lowner}).

The rest of the paper is structured as follows. In Section~\ref{sec:alg}, we
describe the algorithms. In Sections~\ref{sec:compratio}
and~\ref{sec:comp-ratio-sequential}, we provide the competitive ratio analysis
for the online algorithms described in Section~\ref{sec:alg} and determine how
the competitive ratio depends on the smoothed surrogate of the objective
function.  We set up optimization problems for finding surrogates that maximize
the competitive ratio bound derived in these sections.  Numerical examples are
presented in Section~\ref{sec:numerical-exp}. Related works are discussed in
section \ref{sec:related_work}.


\section{Algorithms}
\label{sec:alg}

In this section we describe the class of online algorithms we consider to solve~\eqref{main}, 
and give bounds on the duality gap achieved by these algorithms. 
The algorithms we consider (Algorithms~\ref{algorithm1} and~\ref{algorithm2},
stated in Section~\ref{sec:algostate}) are (modified) greedy algorithms, where
a primal-dual greedy approach is applied to a smoothed surrogate for the
objective.  These algorithms are instances of the general algorithmic framework
for conic online optimization introduced in~\cite{eghbali2016smoothing}, which
can recover state-of-the-art algorithms developed for
Adwords~\cite{buchbinder2007online}, Adwords with concave
returns~\cite{devanur2012online}, and online LP~\cite{buchbinder2009online}.

Before describing the algorithms, it is helpful to slightly rewrite~\eqref{main} and  
state its dual. Throughout, let $G:\mathbf{R}_+\rightarrow \mathbf{R}$ be the indicator function of the interval $[0,b]$, i.e., 
\begin{equation}	
\label{eq:Gdef}
	G(u) = \begin{cases} 0 & \textup{if $u\in [0,b]$}\\ -\infty & \textup{otherwise.}\end{cases}
\end{equation}
Then we can rewrite~\eqref{main} as
\begin{equation}
	\label{eq:main2}
	\underset{x_1,\ldots,x_m\in [0,1]}{\textup{maximize}}\; H\left(\sum_{t=1}^{m}A_t x_t\right) + G\left(\sum_{t=1}^{m}c_tx_t\right).
\end{equation}

To derive a dual program, we rewrite the primal as:
\begin{align*}
	\label{eq:main-rewritten}
	\underset{x_1,\ldots,x_m\in [0,1]}{\textup{maximize}}&\;\;  H\left(U\right) + G\left(u\right). \\
	\textup{subject to} & \;\; U = \sum_{t=1}^{m}A_t x_t \\
	&\;\; u = \sum_{t=1}^{m} c_t x_t	
\end{align*}
Introducing a dual matrix variable $Y$ and dual scalar variable $z$
corresponding to the two equality constraints, we can write the Lagrangian as
follows:
\begin{align*}
&\mathcal{L}(u,U, x_1, \ldots, x_m, Y, z) = H(U) - \inner{Y}{U} + G(u) - \inner{z}{u} +  \sum_{t=1}^{m}(\inner{A_t}{Y} + c_t z) x_t 
\end{align*}
where ${\rm dom} (\mathcal{L}) = \mathbf{S}^n \times \mathbf{R}   \times
[0,1]^{m} \times \mathbf{S}^n \times  \mathbf{R}$. Maximizing $\mathcal{L}$
over all the primal variables, we can derive the dual of~\eqref{main} in terms
of conjugate functions as
\begin{equation}
	\label{eq:main-dual}
 \underset{z,Y}{\textup{minimize}}\; \sum_{t=1}^{m} \left( \inner{A_t}{Y} + c_t z \right)_{+} - H^*\left(Y\right) - G^*\left(z\right),
\end{equation}
where, for a function $\psi: \mathbf{R}^n \rightarrow \mathbf{R}$, $\psi^*$
denotes the concave conjugate of $\psi$, defined as $\psi^*(y) = \inf_{u}
\inner{y}{u} - \psi(u)$. 
The concave conjugate of $G$ is 
\[ G^*\left(z\right) = \begin{cases} b z & \textup{if $z\leq 0$}\\ 0 & \textup{otherwise.}\end{cases}\]
For a trace function $H(U) = \sum_{i=1}^{n}h(\lambda_i(U))$ we have that
$H^*(Y) = \sum_{i=1}^{n}h^*(\lambda_i(Y))$, a result established
in~\cite{lewis1995convex}. 

\subsection{Statement of algorithms}
\label{sec:algostate}
Recall that our objective function $H$ is a monotone, concave, trace function, and that 
$G$ is the concave indicator function of the set $[0,b]$. 
We consider two online algorithms for~\eqref{main}. Algorithm~\ref{algorithm1} makes sequential 
updates to the primal and dual variables, whereas Algorithm~\ref{algorithm2} makes these updates  
simultaneously. These algorithms require a choice of functions $H_S$ and $G_S$ that are 
smooth surrogates for $H$ and $G$. Throughout, we assume these functions satisfy the following.
\begin{assumption}[Assumptions on $H_S$]
\label{assumption-hs}
We assume that $h_S: O\rightarrow\mathbf{R}$ is concave, differentiable, and satisfies $h_S(0)=0$ and $h_S'(0) = h'(0)$, where $O$ is an open interval that contains $\mathbf{R}_+$.
We assume that $H_S:S_+^n\rightarrow \mathbf{R}$ is the trace function $H_S(U) = \sum_{i=1}^{n}h_S(\lambda_i(U))$.
\end{assumption}

\begin{assumption}[Assumptions on $G_S$]
\label{assumption-gs}
We assume that $G_S:O\rightarrow\mathbf{R}$ is concave, differentiable, and satisfies $G_S(0)=0$ and $G_S'(0) = 0$, where $O$ is an open interval that contains $\mathbf{R}_+$.
\end{assumption}

The problem of \emph{designing} $H_S$ and $G_S$, given $H$ and $G$, is the main focus of the paper, 
and is the subject of Section~\ref{sec:compratio}. For now we merely point out that the results of 
this section hold for any choice of $H_S$ and $G_S$ satisfying the basic assumptions above.

\begin{algorithm}[H]
\caption{Sequential Update}
   \label{algorithm1}
   
 Initialize $\hat{z}_0 = G_S' \left(0\right), \hat{Y}_0 = \nabla H_S\left(0\right)$\;
 
\For{$i\leftarrow 1$ \KwTo $m$}{

  Receive $A_t, c_t$\;
  
  $\hat{x}_t = \left\{\begin{array}{cc} 1, &\quad \textup{if $c_t \hat{z}_{t-1} + \inner{A_t}{\hat{Y}_{t-1}}>0$}\\ 0, &\quad \textup{if $c_t \hat{z}_{t-1} + \inner{A_t}{\hat{Y}_{t-1}}\leq 0$;} \end{array}\right.$
 
 $\hat{Y}_{t} = \nabla H_S\left(\sum_{s=1}^{t} A_s  \hat{x}_s \right)$\;
 
 $\hat{z}_{t} = G_S'\left(\sum_{s=1}^{t} c_s \hat{x}_s\right)$\;
 }
\end{algorithm}
\begin{algorithm}[H]
\caption{Simultaneous Update}
   \label{algorithm2}
 
\For{$i\leftarrow 1$ \KwTo $m$}{

	Receive $A_t, c_t$\;	

	{$\displaystyle{\begin{aligned}
	(\tilde{z}_t, \tilde{Y}_t, \tilde{x}_t) & \in
	\displaystyle{\arg\min_{z,Y }\max_{x \in [0,1]}} \;\inner{Y}{A_t x + \textstyle{\sum_{s=1}^{t-1}} A_s \tilde{x}_s}
	+z\left(c_t x + \textstyle{\sum_{s=1}^{t-1}}c_s \tilde{x}_s\right)
		- H^*_S(Y) - G^*_S(z)\;
	\end{aligned}}$}
 	}
\end{algorithm}

In Algorithm \ref{algorithm2}, at each step the primal variable $\tilde x_t$,
and the dual variables $\tilde Y_t$ and $\tilde z_t$, are updated together, by
finding the saddle point of the Lagrangian (solution to the min-max problem).
To argue for the existence of a saddle point, we use Corollary 11.41b in
\cite{rockafellar1998variational} which requires two conditions to be
satisfied. The first condition is that the following  problem has a nonempty,
bounded set of maximizers,
\begin{align*}
\tilde{x}_t \in \arg\max_{x \in [0,1]} H_{S}\left(A_t x + {\sum_{s=1}^{t-1}} A_s \tilde{x}_s\right) + G_{S}\left( x + {\sum_{s=1}^{t-1}}\tilde{x}_s\right),
\end{align*}
which holds because $[0,1]$ is a compact convex set. The second condition requires
\begin{align*}
{\sum_{s=1}^{t-1}} A_s \tilde{x}_s &\in {\rm int}\left(\left\{- A_t x ~\lvert ~ x \in [0,1]\right\}+ {\rm dom} H_S \right),\\
{\sum_{s=1}^{t-1}} c_s \tilde{x}_s &\in {\rm int}\left(\left\{- c_t x ~ \lvert ~ x \in [0,1]\right\} + {\rm dom} G_S \right),
\end{align*}
where the sum is the Minkowski sum of two sets, and $\rm{int}$ denotes the
interior of a set. This condition is also satisfied since 0 is in the interior
of the domains of $h_S$ and $G_S$ (from Assumptions \eqref{assumption-hs} and
\eqref{assumption-gs}).
   
Algorithm~\ref{algorithm1} can also be interpreted in similar terms, by observing that 
the primal $\hat x_t$ update can be written as 
\begin{equation*}
\hat{x}_t \in \argmax_{x \in [0,1]} \; x (c_t \hat{z}_{t-1} + \inner{A_t}{\hat{Y}_{t-1}}),
\end{equation*}
and using the fact that for a differentiable concave function $\psi$, we have $\nabla \psi(u) = \arg\min_{y} \inner{y}{u} - \psi^*(u)$,
the dual updates can be written as
\begin{align*}
\hat{z}_{t} & = \argmin_{z} \;z\left(\textstyle{\sum_{s=1}^{t}} c_s \hat{x}_s\right) - G_S^*\left(z\right)\quad\textup{and}\\
\hat{Y}_{t} & = \argmin_{Y} \; \inner{Y}{\textstyle{\sum_{s=1}^{t}} A_s  \hat{x}_s }- H_S^*\left(Y\right),
\end{align*}
As such, the sequential algorithm can be viewed as alternating over
maximization and minimization of the Lagrangian over primal and dual variables. 

Algorithm~\ref{algorithm1}, assigns integer values to $\hat{x}_t$ and the dual
variable $\hat z_t$ acts as a decision threshold in the assignment rule for
$\hat x_t$.  For the reader familiar with online learning, we point out a
connection with the typical online learning setup.  The $\hat Y_t$ update step
is the same as \emph{Follow-the-Regularized-Leader (FTRL)} update with $H_S$ as
the regularizer, so this algorithm is running FTRL on the dual problem. 

Algorithm~\ref{algorithm2} clearly requires more computation than the
sequential algorithm, but it is also easier to bound its competitive ratio. We
do not focus on implementation details of this algorithm in this paper, but we
point out that one can decide how accurately to solve the saddle-point
subproblem at each iteration, trading off computation with desired accuracy
(for each subproblem).  To analyze this more practical variation of Algorithm
2, our analysis should take into account how solving each saddle-point problem
to $\epsilon$ accuracy affects the overall competitive ratio.  We leave this
additional analysis to future work. 

\subsection{Duality gap}
\label{sec:duality-gap}
The competitive analysis, and subsequent design of $H_S$ and $G_S$ (in
Section~\ref{sec:compratio}), relies on bounding the duality gap.  Define the
dual value achieved by the sequential and simultaneous algorithms
by\vspace{-5pt}
\begin{align*}
	D_{\rm seq} &= \textstyle{\sum_{t=1}^{m}} \left(\inner{A_t}{\hat{Y}_{t-1}} + c_t\hat{z}_{t-1} \right)_{+} - H^*(\hat{Y}_m) - G^*(\hat{z}_m)\\
	D_{\rm sim} &= \textstyle{\sum_{t=1}^{m}} \left(\inner{A_t}{\tilde{Y}_t} +  c_t\tilde{z}_t \right)_{+} - H^*(\tilde{Y}_m) - G^*(\tilde{z}_m).
\end{align*}
The following two lemmas are essentially taken
from~\cite{eghbali2016smoothingfull}. We include proofs in
Appendix~\ref{app:Dgap} to make the present paper self-contained.
Lemma~\ref{lem:duality-gap} bounds the duality gap, and Lemma~\ref{DseqDstar}
relates $D_{\rm seq}$ and $D_{\rm sim}$ to the dual optimal value $D^\star$,
given an additional assumption on the gradient of $H_S$. 

 \begin{lemma}\label{lem:duality-gap}
	Let $\tilde{x}_t, \tilde{Y}_t$, and $\tilde{z}_t$ and $\hat{x}_t$, $\hat{Y}_t$, and $\hat{z}_t$ 
	denote the $t$th iterate of the simultaneous and sequential algorithms, respectively.
	Then the duality gaps for the two algorithms satisfy the lower bounds 
	\begin{align}\label{duality-gap}
		H_S\left(\textstyle{\sum_{t=1}^{m}}A_t\tilde{x}_t\right) + G_S\left(\textstyle{\sum_{t=1}^{m}}c_t\tilde{x}_t\right) 
		- D_{\rm sim}  &\geq H^*(\tilde{Y}_m) + G^*\left(\tilde{z}_m\right).\\  
		H_S\left(\textstyle{\sum_{t=1}^{m}}A_t\hat{x}_t\right) + G_S\left(\textstyle{\sum_{t=1}^{m}}c_t\hat{x}_t\right) 
		- D_{\rm seq}  &\geq H^*(\hat{Y}_m) + G^*\left(\hat{z}_m\right) + \nonumber\\
		 &\qquad \sum_{t=1}^{m} \inner{ A_t \hat{x}_t}{\hat{Y}_{t} - \hat{Y}_{t-1}} +  
		\sum_{t=1}^{m} \inner{ c_t\hat{x}_t}{ \hat{z}_{t} - \hat{z}_{t-1}}.  \label{duality-gap-seq}
	\end{align}
 \end{lemma}
\begin{proof}
	See Appendix~\ref{app:Dgap}. 
\end{proof}

This duality gap bound is reminiscent of regret bounds in online learning. In
regret analysis one directly bounds the right hand side of \eqref{duality-gap-seq} in terms of the horizon
and problem parameters. However, bounding the competitive ratio requires more
assumptions and further analysis. The PSD Diminishing Returns (PSD-DR) property 
is used to relate $D_{\rm seq}$ and $D_{\rm sim}$ to the dual optimal value $D^\star$.
\begin{lemma}\label{DseqDstar}
	If $H_S$ satisfies PSD-DR, then
	$D_{\rm seq} \geq D^\star$ and 
	$D_{\rm sim} \geq D^\star$.
\end{lemma}
For completeness the proof is given in Appendix~\ref{app:Dgap}.  We remark that
the PSD-DR property is a special case of the property used in the abstract
framework of~\cite{eghbali2016smoothing} that was applied to general concave
functions on proper cones. In order to develop computational methods to
\emph{design} the function $H_S$ (in Section~\ref{sec:smooth-h}), we will need
to impose the PSD-DR property on $H_S$ in a computationally tractable way. A
central observation of this paper is that this is possible for the PSD-DR
property by exploiting a celebrated theorem of L\"owner. We discuss this in
detail in Section~\ref{sec:smooth-h}.

\section{Competitive ratio analysis}
\label{sec:compratio}
In this section we show how to design $H_S$ and $G_S$ in the sequential algorithm (Algorithm~\ref{algorithm2}) to maximize
a lower bound on the competitive ratio achieved. We discuss the appropriate modifications for the sequential 
algorithm in Section~\ref{sec:comp-ratio-sequential}. In what follows, $e = 2.718\ldots$ is Euler's number.

We first state a technical lemma (Lemma \ref{comp-ratio-pre}) that relates the 
budget consumed and competitive ratio achieved by Algorithm~\ref{algorithm2} to the 
functions $G_S$ and $H_S$.  
The only dependence on the data (the values of $A_t$ and $c_t$)
is via two parameters, $\theta$ and $\Theta$, which are lower and upper 
bounds on $c_t^{-1}\tr(A_t)$ for all $t$. 
We use a parameter $\gamma \geq 1$ as a trade-off parameter between the budget consumption and competitive ratio. The result says that if $H_S$, $G_S$  satisfy certain inequalities that involve $\gamma$ and $\beta$, 
then Algorithm~\ref{algorithm2} achieves a competitive ratio of at
least $ \frac{1}{\gamma/ (e-1) + \beta}$ using a budget
of $G_S'^{-1}(-h'(0) \Theta)$ (the inverse
function of the derivative of $G_S$ evaluated at a point). Concrete values for
these bounds are worked out at the end of this section.
This is a similar style of result to~\cite{buchbinder2009online}, which applies 
to a special class of linear programs. 

\begin{lemma}
\label{comp-ratio-pre}
	Let $G$ be as defined in~\eqref{eq:Gdef}, and let $h$ be concave and monotonically increasing, 
	with corresponding trace function $H$. Let $G_S$ satisfy Assumption~\ref{assumption-gs}, 
	$h_S$ satisfy Assumption~\ref{assumption-hs}, and suppose that the corresponding trace 
	function $H_S$ satisfies the PSD-DR property. Suppose that 
	$\theta \leq c_t^{-1}\tr(A_t) \leq \Theta$ for all $t$ and 
	$u_{\max} \geq \lambda_{\max}\left(\sum_{t=1}^{m}A_t\tilde{x}_t\right)$.\footnote{Note 
	that we could choose, for instance, 
	$u_{\max} = b'\max_{t} c_t^{-1}\lambda_{\max}(A_t)$, but for certain classes of problems better bounds 
	may be available.} Then
	\begin{enumerate}
	\item If for a given $\gamma \geq 1$, $G_S$ satisfies
	\begin{align}
	 \gamma G_S(u) &\leq G^*(G_S'(u)) + \frac{\gamma}{e-1}h(\theta u)\quad\textup{for all $u\in [0,\infty)$}\label{eq:G_gamma},
	 \end{align}
	and there exists $\beta >0$ such that $h_S$ satisfies
	\begin{align}
	\gamma h_S(u) &\leq h^*(h_S'(u)) + \beta h(u)\quad\textup{for all $u\in [0,u_{\max}]$}\label{eq:h_beta},
	\end{align}
	then the iterates $\tilde{x}_1,\ldots,\tilde{x}_m$ of Algorithm~\ref{algorithm2} satisfy
	\begin{equation}
	\label{eq:crbound}
	 \textstyle{H\left(\sum_{t=1}^{m}A_t\tilde{x}_t\right) \geq \frac{1}{\gamma/(e-1)+\beta} D^\star.}
	\end{equation}
	\item The iterates $\tilde{x}_1,\ldots,\tilde{x}_m$ of Algorithm~\ref{algorithm2} also satisfy
	\begin{equation}
	\label{eq:bprime}
	 \sum_{t=1}^{m}c_t\tilde{x}_t \leq b':= \inf\{u: G_S'(u) \leq -h'(0)\Theta\}.
	\end{equation}
	
	\end{enumerate}
\end{lemma}
Before providing a proof of Lemma~\ref{comp-ratio-pre}, we explain how we use
it in what follows.  As mentioned before, we consider $\gamma$ a design
parameter that trades off between (possible) budget violation and competitive
ratio achieved. The smaller the parameter $\gamma$, the larger the competitive
ratio (evident from \eqref{eq:h_beta} and \eqref{eq:crbound}) and the larger
the budget violation.  The role of $\gamma$ in the amount of budget consumed
and the reason for requiring $\gamma \geq 1$ is discussed in
Section~\ref{sec:smooth-g}.  In Section~\ref{sec:smooth-g} we use
\eqref{eq:G_gamma} and \eqref{eq:bprime} in Lemma~\ref{comp-ratio-pre} to
design $G_S$ to minimize the budget $b'$ consumed for a given $\gamma$. This
design problem requires the parameter $\theta$ only. The parameter $\gamma$
allows us to decouple the design of $G_S$ and $h_S$. 

In Section~\ref{sec:smooth-h}, we use \eqref{eq:h_beta} and \eqref{eq:crbound}
in Lemma~\ref{comp-ratio-pre} to design $H_S$ via solving a convex optimization
problem that maximizes the bound on competitive ratio (by minimizing $\beta$).
This design problem requires the parameter $u_{\max}$ only.

\begin{proof}[{Proof of Lemma~\ref{comp-ratio-pre}}]
	First we show that $\sum_{t=1}^{m}c_t\tilde{x}_t \leq b'$. To do this we use the optimality conditions of 
	the saddle point problem in Algorithm~\ref{algorithm2}. Indeed, for any $t$, we have that 
	\[ \tilde{Y}_t = \nabla H_S\left(\textstyle{\sum_{s=1}^{t}}A_s \tilde{x}_s\right)\quad\textup{and}\quad
	\tilde{z}_t = G_S'\left(\textstyle{\sum_{s=1}^{t}c_s \tilde{x}_s}\right),\]
	and that $\inner{\tilde{Y}_t}{A_t} + c_t\tilde{z}_t < 0$ implies $\tilde{x}_t = 0$. 
	 
	Arguing by contradiction, let $T$ be the smallest index such that 
	$\sum_{t=1}^{T}c_t\tilde{x}_t > b'$. Then, by the definition of $b'$ we have that 
	$\tilde{z}_T = G_S'\left(\sum_{t=1}^{T}c_t\tilde{x}_t\right) < -h'(0)\Theta$.
	Observe that $\sum_{t=1}^{T}A_t\tilde{x}_t \succeq 0$, from which it follows, from the fact that $H_S$ is PSD-DR,
	that $\nabla H_S\left(\sum_{t=1}^{T}A_t\tilde{x}_t\right) \preceq \nabla H_S(0) = h_S'(0)I$. Then
	\begin{equation*}
 	\inner{\tilde{Y}_T}{A_T} + c_T\tilde{z}_T  =
	\inner{\nabla H_S\left(\textstyle{\sum_{t=1}^{T}}A_t \tilde{x}_t\right)}{A_T} + c_T \tilde{z}_T
	\leq h_S'(0)\tr{A_T} + c_T\tilde{z}_T < 0
	\end{equation*}
	where, for the last inequality, we use the fact that $h'_S(0) = h'(0)$ and that $c_T^{-1}\tr{A_T} \leq \Theta$. 
	It follows from the optimality conditions that $\tilde{x}_T =0$. But then 
	$\sum_{t=1}^{T-1}c_t\tilde{x}_t > b'$, contradicting our choice of $T$. It follows that $\sum_{t=1}^{m}c_t\tilde{x}_t \leq b'$.  

	We now turn our attention to the bound on the competitive ratio. 
 	Let $U = \sum_{t=1}^{m} A_t \tilde{x}_t$ and $u = \sum_{t=1}^{m}c_t \tilde{x}_t$, so that 
	$\tilde{Y}_m = \nabla H_S(U)$, and $\tilde{z}_m = G_S'(u)$. 
	By the duality gap bound (Lemma~\ref{lem:duality-gap}) we have
	\begin{align}\label{duality-gap}
		H_S\left(U\right) + G_S\left(u\right) - {D_{\rm sim}} \geq  H^*(\tilde{Y}_m) + G^*\left(\tilde{z}_m\right) .
	\end{align}
	By the primal allocation rule in Algorithm \ref{algorithm2}, we have 
		$\tilde{x}_t \left( c_t \tilde{z}_{t} + \inner{A_t}{\tilde{Y}_{t}}\right) \geq 0$.
	Combining this observation with the concavity of $H_S$ and $G_S$, we get
	\begin{align*}
		H_S\!\left(\textstyle{\sum_{s=1}^{t}} A_s\tilde{x}_s\!\right) +
		G_S\!\left(\textstyle{\sum_{s=1}^{t}}c_s \tilde{x}_s\!\right) - 
		H_S\!\left(\textstyle{\sum_{s=1}^{t-1}}A_s\tilde{x}_s\!\right) - 
		G_S\!\left(\textstyle{\sum_{s=1}^{t-1}}c_s \tilde{x}_s\!\right) \\
		 \geq \tilde{x}_t\!\left(c_t \tilde{z}_{t} + \inner{A_t}{\tilde{Y}_{t}}\right) \!\geq 0.
	\end{align*}
	By taking the sum over $t$ and telescoping the sum we get
	\begin{align}\label{eq:increasing-seq}
	H_S\left(\textstyle{\sum_{s=1}^{m}} A_s\tilde{x}_s\right) + G_S\left(\textstyle{\sum_{s=1}^{m}} c_s\tilde{x}_s\right)\geq 0.
	\end{align} 
	The last preparatory observation we need is the inequality $H(U)\geq h(\theta u)$. To see why this holds, note that 
	because $h$ is monotonically increasing, and $\theta \leq c_t^{-1}\tr(A_t)$ for all $t$,  
	\begin{align}
	 h(\theta u) = h\left(\theta\textstyle{\sum_{t=1}^{m}}c_t\tilde{x}_t\right) &\leq h\left(\textstyle{\tr(\sum_{t=1}^{m}}A_t\tilde{x}_t)\right)\nonumber\\
	& = h(\tr(U))\leq \sum_{i=1}^{n}h(\lambda_i(U)) = H(U)\label{eq:theta-ineq}
	\end{align}
	where the last inequality holds because $h$ is concave and $h(0) \geq 0$ so $h$ is subadditive on $[0,\infty)$.
	Now we can write
	\begin{align*}
		&H\left(U\right)  - {D_{\rm sim}} \\
		 &\geq -H_S\left(U\right) - G_S(u) + H^*\left(Y\right) + G^*\left(z\right) + H\left(U\right) 
			&&\textup{ By~\eqref{duality-gap}} \\
		&\geq -H_S(U) + \left(\gamma-1\right) G_S(u) + \textstyle{\left(1 - \frac{\gamma}{e-1}\right)} H(U) +  H^*(Y)  
			&&\textup{ By~\eqref{eq:G_gamma} and \eqref{eq:theta-ineq}} \\ 
		&\geq -H_S(U) + \left(1-\gamma\right) H_S(U) + \textstyle{\left(1 - \frac{\gamma}{e-1}\right)} H(U) +  H^*(Y) 
			&&\text{ By \eqref{eq:increasing-seq} and $\gamma \geq 1$} \\
		& = -\gamma H_S(U) + \textstyle{\left(1 - \frac{\gamma}{e-1}\right)} H(U) +  H^*(Y)  \\
		&\geq \textstyle{\left(1-\frac{\gamma}{e-1} -\beta\right)} H(U)  &&\text{ By~\eqref{eq:h_beta}}.
	\end{align*}
	Then the result follows from Lemma \ref{DseqDstar}. 
 \end{proof}

These results allow us to search for functions $h_S$ and $G_S$ that satisfy the assumptions of Lemma~\ref{comp-ratio-pre}. 
We use $\gamma$ as a design parameter, controlling the trade-off between competitive ratio and the (possible) budget violation. 

\subsection{Smoothing the budget penalty function $G$, taking $h$ into account.}  
\label{sec:smooth-g}
For a fixed choice of $\gamma$, our aim is
to design $G_S$ so as to minimize $b'$. We do so by explicitly constructing 
$G_S'(u)$ such that~\eqref{eq:G_gamma} is satisfied with equality for all $u\geq 0$, and showing that such a $G_S'$ is optimal, 
in the sense that for a given $\gamma\geq 1$, defining $G_S$ via~\eqref{G_S} minimizes the bound $b'$ on 
the budget consumed by the simultaneous algorithm.
\begin{proposition}
Let $\bar{G}_S$ be any function that satisfies~\eqref{eq:G_gamma}. Let \footnote{We can extend the domain of $G_S$ to negative reals by letting $G_S  = 0$ on $\mathbf{R}_{-}$ to satisfy the technical assumption on the domain of $G_S$ in Assumption \eqref{assumption-gs}.}
\begin{align}\label{G_S}
G_S'(u) =  - \frac{\gamma}{b (e-1)}
	\int_{0}^{u} \exp\left(\frac{\gamma}{b} (u-v)\right) \theta h'\left(\theta v\right) \; d v \quad \textup{for  } u \geq 0 
\end{align}
Then $\bar{G}_{S}'(u) \geq G_S'(u)$ for all $u\geq 0$ and so
\[ \inf\{u: G_S'(u) \leq -h'(0)\Theta\} \leq \inf\{u: \bar{G}_S'(u) \leq -h'(0)\Theta\}.\]
\end{proposition}
\begin{proof}
	Clearly $G_S'(u) \leq 0$ for all $u\geq 0$, and so for any $u\geq 0$, we have that $G^*(G_S'(u)) = bG_S'(u)$.
	We now explain why $G_S'$ satisfies~\eqref{eq:G_gamma} with equality. (This can also be verified by direct substitution.) 
	Because $G_S(0)=0$ (by assumption), it is enough to show 
	that $G_S'$ satisfies the linear, constant coefficient differential equation obtained by equating, and then differentiating, 
	both sides of~\eqref{eq:G_gamma}. After rearranging terms, this differential equation is
	\begin{equation*}
	 G_S''(u) = \frac{\gamma}{b} G_S'(u) - \frac{\gamma}{b(e-1)}\theta h'(\theta u).
	\end{equation*}
	The function $G'_S(u)$ that satisfies this equation for $u\in [0,\infty)$ together with  
	$G'(0)=0$, is given by the convolution of $-\frac{\gamma}{b(e-1)}\theta h'(\theta u)$ with $e^{\gamma u/b}$, as in~\eqref{G_S}.

Since $G_S'$ satisfies~\eqref{eq:G_gamma} with equality, it follows from Gronwall's 
inequality (see \cite{dragomir2003some} Corollary 2) 
that for any other function $\bar{G}_S$ that satisfies \eqref{eq:G_gamma}, we have 
$\bar{G}_{S}'(u) \geq G_S'(u)$ for all $u\geq 0$.
\end{proof}

\begin{remark}
Note that the function $G_S$ is a convolution of $h'(\theta u)$ with an
exponential function.  It can also be viewed as a function derived by applying
Nesterov's smoothing technique to $G$. To show this, we have to argue that
$G_S^* - G^*$ is a concave function. Note that $G^*$ is linear over
$\mathbf{R}_{-}$; therefore, $G_S^* - G^*$ is concave on the positive reals
with $G_S^*(0) - G^*(0) = 0$, and for  $y > 0$, $G_S^*(y) - G^*(y) = 0$. This
establishes the concavity of $G_S^* - G^*$. As we discuss in Example 1 below,
$G_S^* - G^*$ has a closed form when $h$ is a linear function. However, for
more general examples of $h$ we do not have a closed form for $G_S^* - G^*$.
\end{remark}
Now, to justify the fact that we require $\gamma \geq 1$, we find a lower bound $G_S'$ given in \eqref{G_S} in terms of $\gamma$:
	\begin{align*}
		G_S'(u) &= - \frac{\gamma}{b (e-1)} \exp\left(\frac{\gamma}{b} u\right) \int_{0}^{u} \exp\left(-\frac{\gamma}{b} v\right) \theta h'  \left(\theta v\right) \; d v \\
		&\geq -  \frac{\gamma \theta h'\left(0\right)}{b (e-1)}\exp\left(\frac{\gamma}{b} u\right) \int_{0}^{u} \exp\left(-\frac{\gamma}{b } v\right) \; d v =  -\frac{\theta h'\left(0\right) (\exp(\frac{\gamma}{b} u) - 1)}{e-1}
	\end{align*}
The inequality above is exact if $h$ is linear, e.g., $H(U) = \tr(U)$. Now we can bound $b'$ as:
\begin{align*}
b'= \inf\{u: G_S'(u) \leq -h'(0)\Theta\} \geq \frac{b}{\gamma} \log\left(\frac{\Theta}{\theta} (e-1) + 1\right)
\end{align*}
This shows that if $\gamma < 1$, then even in the trivial case where $\theta = \Theta$ 
and $h$ is linear the algorithm can go over budget.

We now compute $G_S'$, and bound the budget $b'$ consumed by the simultaneous algorithm using $G_S'$, for three examples.
 \begin{example}[Linear objective function]
\label{eg:linearupper}
Consider the linear function $h(u) = u$ which
translates to $H(U) =  \tr(U)$. This case allows us to show that our approach specializes 
to recover known results. In this case, the problem reduces to a linear
program with one budget constraint,
\begin{equation*}
 \mbox{maximize} \quad  \sum_{t=1}^{m} \tr(A_t) x_t\quad \mbox{subject to}  \quad 
	\begin{cases}\sum_{t=1}^{m} c_t x_t \leq b&\\
	0 \leq x_t \leq 1 & \textup{for $t = 1,2,\ldots, m$.}
	\end{cases}
\end{equation*}
Computing the integral~\eqref{G_S} gives 
$G_S'(u) = \theta \left(1- \exp\left(\frac{\gamma}{b} u\right)\right)/(e-1)$.
This choice of $G_S'$ corresponds to the exponential update algorithm for online LP~\cite{buchbinder2009online}, 
and, in this particular case, $G_S$ can also be derived as a smooth surrogate for $G$ using Nesterov smoothing with 
a shifted entropy as the proximity function \cite{eghbali2016smoothing}. In this case the bound $b'$ on the budget consumed
is given by
\begin{align*}
b'=\inf\left\{u~|~G_S'(u) <  - h'(0)\Theta\right\}
= \frac{b}{\gamma} \log\left((e-1)\frac{\Theta}{\theta} + 1 \right).
\end{align*}
Choosing 
$\gamma \geq \log\left((e-1)\frac{\Theta}{\theta} + 1\right)$
ensures that the budget is not violated.
\end{example}
\begin{example}[D-optimal experiment design]  Suppose that 
\label{eg:dupper}
$h(u) = \log(u+1)$, the objective function of interest in D-optimal experiment design. 
In this case $h'(u) = (1+u)^{-1}$, and we note that it is possible
to express the optimal $G_S'$ from~\eqref{G_S} in terms of special functions called \emph{exponential integrals}. 
To obtain an upper bound on $b'$, and hence on the budget consumed, we use the fact that 
$h'(\theta v) = (1+\theta v)^{-1} \geq e^{-\theta v}$ whenever $\theta v\geq 0$. 
Using this inequality in~\eqref{G_S}, and computing the resulting (elementary) integral, we obtain the bound
\[ G_S'(u) \leq \frac{1}{(e-1)}\frac{e^{-\theta u} - e^{\frac{\gamma}{b}u}}{\theta^{-1} + (b/\gamma)}
	\leq \frac{1}{(e-1)}\frac{1-e^{\frac{\gamma}{b}u}}{\theta^{-1} + (b/\gamma)}.\]
Rearranging and using the definition of $b'$ we see that 
\begin{equation}	
\label{eq:bprimeDbound}
	b' \leq \frac{b}{\gamma}\log\left[(e-1)\Theta\left(\theta^{-1}+  b/\gamma\right) + 1\right].
\end{equation}
Since $\gamma\geq 1$, we can ensure that there is no budget violation, i.e.,  $b'\leq b$, by choosing
$\gamma \geq \log\left[(e-1)\Theta \left(\theta^{-1}+b\right) + 1\right]$.
Figure \ref{nwa-1} shows examples of $G_S$ for two values of $\gamma$ when $h(u) = \log(1+u)$. 
\end{example}

 \begin{figure}[h]  
  \begin{subfigure}{0.47\textwidth}
\centering    
                   \includegraphics[width=0.95\textwidth]{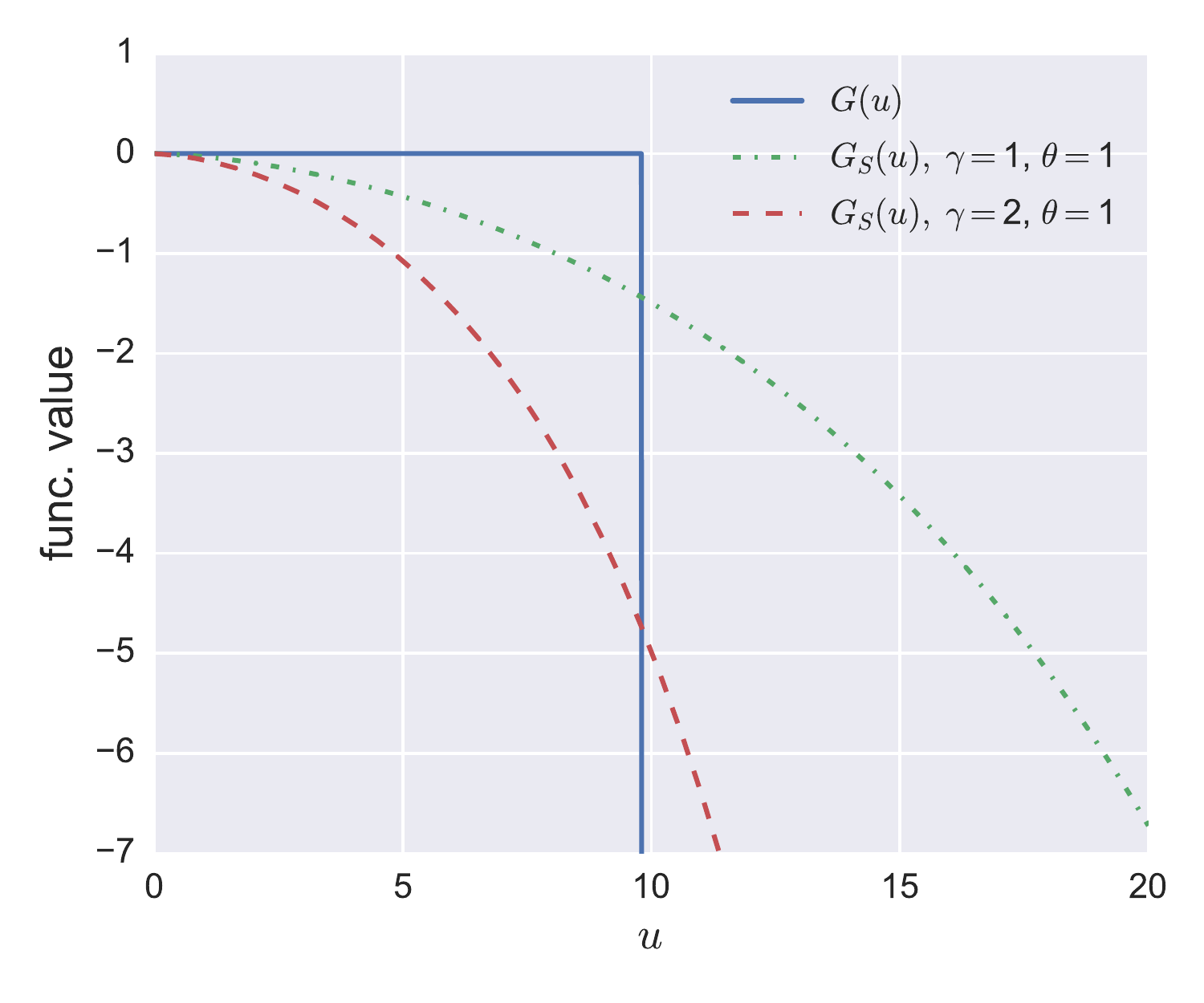}        
          
                     \caption{Examples of $G_S$ when $h(u)= \log(u+1)$}
                     \label{nwa-1}
 \end{subfigure}
\begin{subfigure}{0.47\textwidth}
\centering    
                   \includegraphics[width=0.95\textwidth]{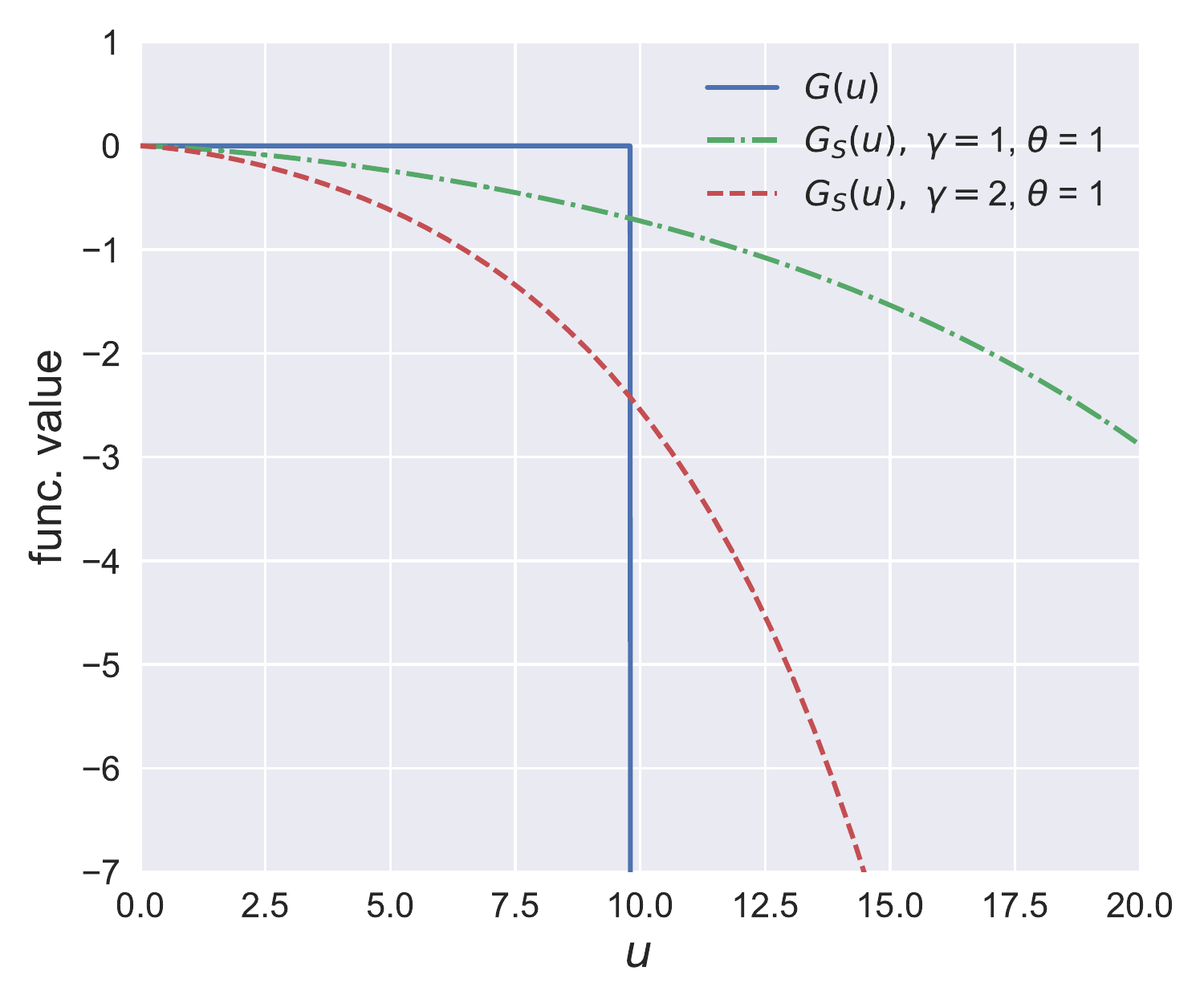}        
           
			\caption{Example of $G_S$ when $h(u) = 1-\frac{1}{u+1}$}
			\label{nwa-2}	
 \end{subfigure}
 \caption{Examples of $G_S$ given in \eqref{G_S} for different values of $\gamma$ and $\theta$.} 

      \end{figure}

\begin{example}[A-optimal experiment design] 
\label{eg:aupper}
When $h(u) = 1 - \frac{1}{1+u}$, and so $h'(u) = (1+u)^{-2}$, it
is again possible to express the optimal $G_S'$ in terms of exponential integrals. Using the same strategy as example~\ref{eg:dupper}, 
but with the bound $h'(\theta v) = (1+\theta v)^{-2} \geq e^{-2\theta v}$ (for $v \geq 0$), gives
\[ G_S'(u)  
	\leq \frac{1}{(e-1)}\frac{1-e^{\frac{\gamma}{b}u}}{\theta^{-1} + (2b/\gamma)}\;\;\textup{and}\;\;
b' \leq \frac{b}{\gamma}\log\left[(e-1)\Theta(\theta^{-1} + 2b/\gamma)+1\right].\]
Figure~\ref{nwa-2} shows examples of $G_S$ for two values of $\gamma$ when $h(u) = 1-\frac{1}{1+u}$. 
\end{example}

\subsection{Smoothing $h$.} 
\label{sec:smooth-h}
We are now in a position to design $h_S$, a smoothed surrogate for $h$.  Recall
that, for a fixed choice of $\gamma\geq 1$, we can use~\eqref{G_S} to design an
optimal $G_S'$. This gives a bound $b'$ on the budget consumed by
Algorithm~\ref{algorithm2}.  We design $h_S$ by solving the following
optimization problem over $\beta$ and the function $h_S$:
\begin{equation}
	\label{eq:smooth-h1}
	 \textup{minimize}\; \beta\;\;\textup{subject to}\; \begin{cases}
	\gamma h_S(u) \leq h^*(h_S'(u)) + \beta h(u) & \forall u\in [0,u_{\max}]\\
	h_S\;\;\textup{satisfies PSD-DR.}\end{cases}
\end{equation}
Note that this problem comes directly from Lemma~\ref{comp-ratio-pre} and the fact that for fixed $\gamma$, we can maximize 
the bound on the competitive ratio, $(\gamma/(e-1) + \beta)^{-1}$, by minimizing $\beta$. 

The key challenge in solving this optimization problem is imposing the PSD-DR
property.  A crucial observation in this paper is that, for a differentiable
trace function, satisfying PSD-DR can be expressed in an equivalent, more
explicit form.  This is an easy consequence of \emph{L\"owner's theorem} for
matrix monotone functions~\cite{Hansen2013}, an important result in matrix
analysis.

\begin{proposition}
\label{prop:lowner}
Suppose $H_S(U) = \sum_{i=1}^{n}h_S(\lambda_i(U))$ where $h_S$ is a monotone increasing, concave function, and $h_S(0) = 0$. 
Then $H_S$ is PSD-DR for all $n$, if and only if there exists a positive measure $\mu$ supported on $[0,1]$ such that 
\begin{gather}\label{eq:lowner}
	h_S(u) = \int_{0}^{u} y(u') \; d u' \quad\textup{where}\quad
	y(u) = \int_{0}^{1} \frac{1}{u \lambda + (1-\lambda)} \; d \mu(\lambda).
\end{gather}
\end{proposition}
\begin{proof}
	This follows from~\cite[Theorem 4.9]{Hansen2013}. 
\end{proof}

This alternative description allows us to impose the constraint that $h_S$
satisfies PSD-DR  in the optimization problem for designing $h_S$. We now
rewrite this optimization problem in a more computationally useful form.  We
use the change of variable 
$h_S\left(u\right) =  \int_{0}^{u} y\left(s\right)\; d s$, introduce 
$\mu$ from~\eqref{eq:lowner} as a decision variable, and express $y$ in terms 
of $\mu$. Doing so we obtain the following optimization problem. 
 \begin{align}\label{cont-problem}
  \mbox{minimize}_{\beta,y,\mu}&{ \quad  \beta}\\ \notag
 \mbox{subject to}& \quad  \gamma \int_{0}^{u} y\left(s\right) ds  - h^*\left(y\left(u\right)\right) \leq \beta  h\left(u\right) \qquad \forall u \in [0,u_{\max}]\\ \notag 
 & y\left(t\right) = \int_{0}^{1} \frac{1}{t \lambda + \left(1-\lambda\right)} \; d \mu\left(\lambda\right)\\
 &\quad \textup{$\mu$ a positive measure supported on $[0,1]$.} \notag
 \end{align}
We denote by $\beta(\gamma)$, the optimal value of~\eqref{cont-problem} for a
given $\gamma$. We extend $h_S$ linearly on $\mathbf{R}_{-}$ to satisfy the
technical assumption on the domain of $h_S$ in Assumption
\eqref{assumption-hs}.

 \begin{figure}
  \begin{subfigure}{0.46\textwidth}
\centering    
                \includegraphics[width=0.85\textwidth]{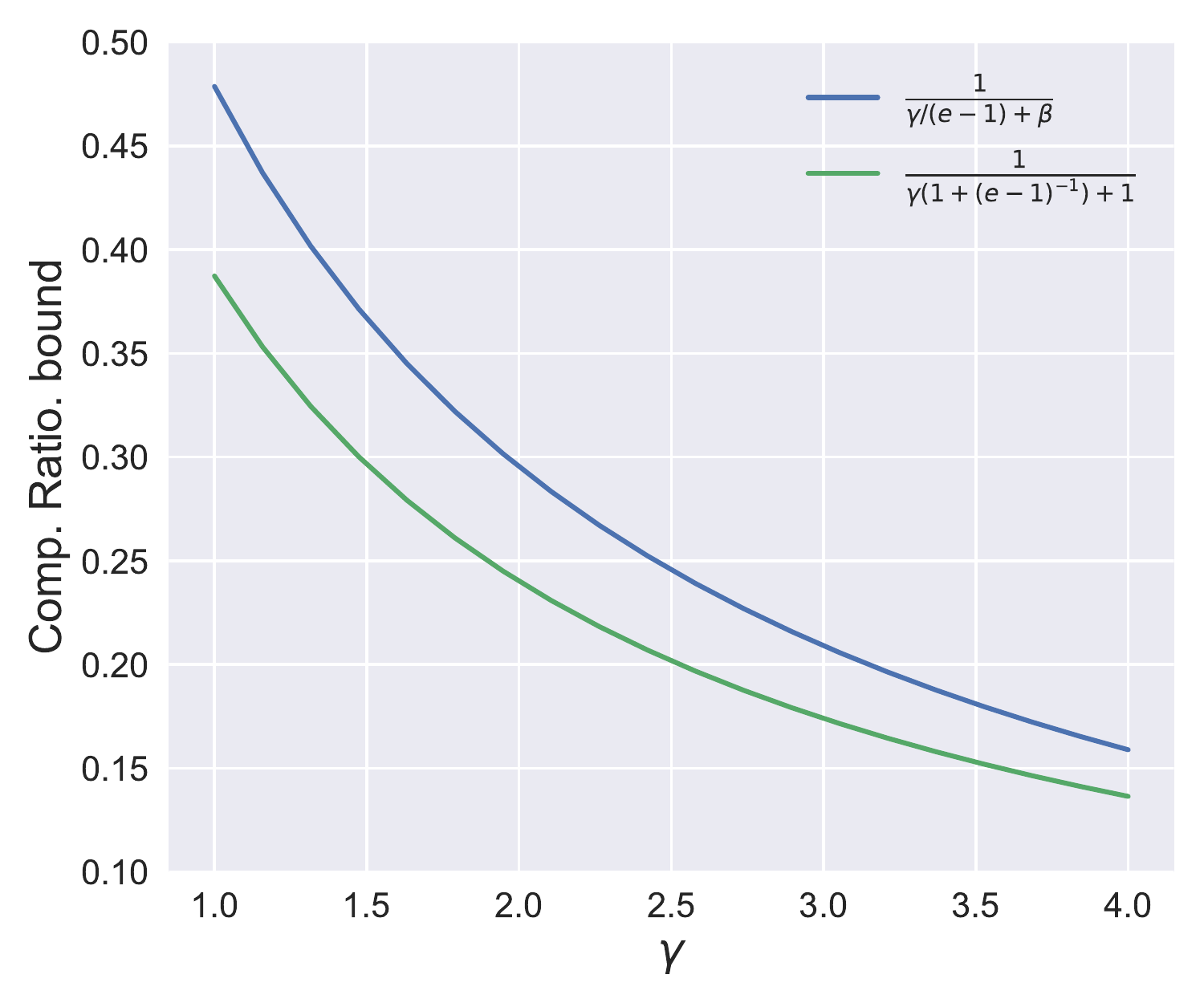}        
            \caption{}
              \label{comrat_gamma_log}     
 \end{subfigure}
  \begin{subfigure}{0.46\textwidth}
\centering    
                \includegraphics[width=0.85\textwidth]{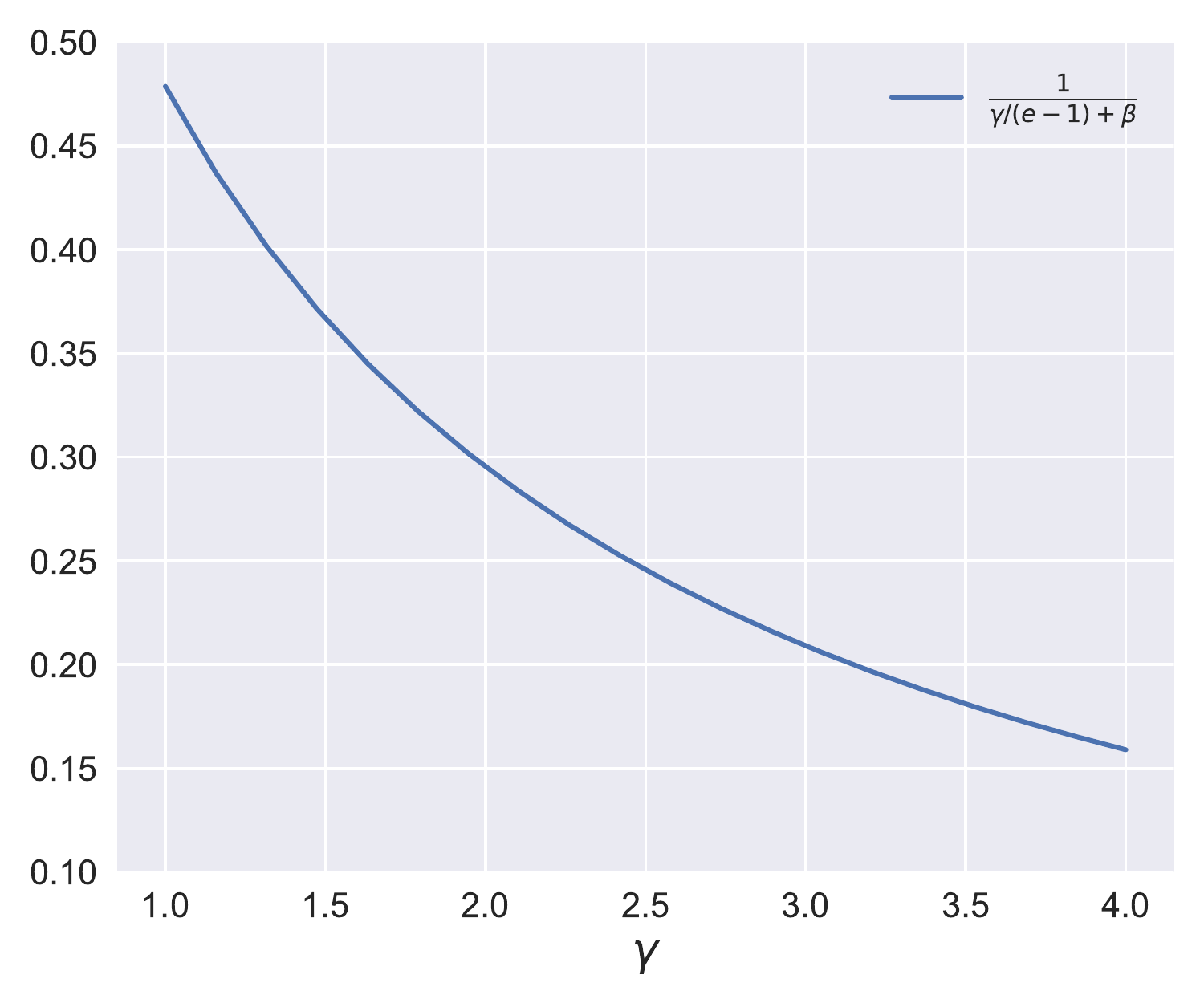}        
            \caption{}
              \label{comrat_gamma_inv}     
 \end{subfigure}
\caption{(a) Two competitive ratio bounds vs $\gamma$ for $h(u) = \log(u+1)$.
The first bound is achieved by numerically finding the optimum $\beta(\lambda)$
in \eqref{eq:smooth-h1}, while the second bound uses the fact that
$\beta(\lambda) \leq \lambda + 1$ when $h(u) = \log(u+1)$. (b) Competitive
ratio bound achieved by numerically finding the optimum $\beta(\lambda)$  for
$h(u) = 1 - 1/(u+1).$}
\end{figure}

\begin{thm}\label{comp-ratio}
	Suppose that $G_S$ is defined as in~\eqref{G_S}, 
	and $H_S$ is the trace function corresponding to $\int_{0}^uy(s)\;ds$, 
 	where $y$ is optimal for~\eqref{cont-problem}. Then the 
	iterates $\tilde{x}_1,\ldots,\tilde{x}_m$ of Algorithm~\ref{algorithm2} satisfy
 	\begin{align*}
 		&\textstyle{H\left( \sum_{t=1}^{m} A_t \tilde{x}_t \right) \geq  \frac{1}{\gamma/ (e-1) + \beta(\gamma)} D^\star}
\quad\textup{and}\quad
 	\sum_{t=1}^{m} c_t\tilde{x}_t \leq G_S'^{-1}(-h'(0) \Theta)
 	\end{align*}
	where $\Theta\geq c_t^{-1}\tr(A_t)$ for all $t$
and $\beta(\gamma)$ is the optimal value of $\beta$ in~\eqref{cont-problem}. 
Moreover, if $H$ satisfies the PSD-DR property, then $\beta(\gamma) \leq \gamma + 1$. 
\end{thm}
\begin{proof}
The proofs of the first two inequalities follow directly from Lemma
\ref{comp-ratio-pre}.  To see that $\beta(\gamma) \leq \gamma + 1$ when $H$
satisfies the PSD-DR property, one can note that $y(u) = h'(u)$ and $\beta =
\gamma + 1$ is a feasible solution for \eqref{cont-problem} in that case.
\end{proof}

Note that if $H_S$ satisfies PSD-DR, then $H$ is a feasible solution to problem
\eqref{cont-problem}; however, it is not necessarily the optimal solution. As
we will see in the next section, the objective function in D-optimal experiment
design, $H(U) = \log\det(I+U)$, provides such and example, where the
competitive ratio improves as the result of smoothing even though $H$ is
PSD-DR. In the next section, we also consider the A-optimal experiment design
where the objective function does not satisfy PSD-DR.  

 \subsubsection{Examples}
 \begin{example}[Linear objective] When $h(u) = u$ then 
\[ h^*(y) =\begin{cases}  0 & \textup{if $y=1$}\\ -\infty & \textup{otherwise.}\end{cases}\]
Therefore, the only feasible solution to~\eqref{eq:smooth-h1} is $h_S(u) = u$.
This is a PSD-DR function.  The corresponding value of $\beta$ is $\beta =
\gamma$. As such, the competitive ratio achieved by Algorithm~\ref{algorithm2}
is at least $\frac{1}{\gamma}\left(1-\frac{1}{e}\right)$. 
\end{example}
 
 \begin{example}[D-optimal experiment design]
In this case $h\left(u\right) = \log\left(u+1\right)$, i.e., $H(U) =
\log\det(I+U)$. Note that $H$ satisfies the PSD-DR property, because we can
write $h'(u) = 1/(1+u)$ in the form of~\eqref{eq:lowner} by choosing the
measure $\mu(\lambda)$ to have mass $1/2$ at $\lambda = 1/2$. By putting
$h_S(u) = h(u)$, we can conclude that $\beta(\gamma) \leq 1+\gamma$.  This
means that the competitive ratio achieved by Algorithm~\ref{algorithm2} is at
least $\frac{1}{\gamma(1-1/e)^{-1} + 1}$. 

By solving~\eqref{cont-problem} computationally, we can design an $h_S$ that
achieves a better competitive ratio.    Figure~\ref{log_cont_problem_11} shows
the solution of \eqref{cont-problem} for $\gamma = 1$ over the finite horizon
$[0, u_{\max}] = [0,1]$.  Figures~\ref{log_cont_problem_41}
and~\ref{log_cont_problem_410} show the solution of~\eqref{cont-problem} for
$\gamma = 4$ and $u_{\max}= 1$ and $u_{\max} = 10$ respectively. We note that
the optimal measures are quite complicated---in the case $\gamma=1$ the optimal
measure seems to be atomic, whereas in the case $\gamma = 4$, it has a
qualitatively different structure.  In Figure \ref{comrat_gamma_log}, we plot
the competitive ratio bound of Theorem \ref{comp-ratio} vs $\gamma$, and
compare it with the bound obtained without smoothing $h$.  We can see that
$\gamma$ captures the trade off between the budget violation and competitive
ratio.  A smaller $\gamma$ gives a better competitive ratio at the expense of
larger budget violation, quantified by the bound given in
Example~\ref{eg:dupper}.
\end{example}

 \begin{example}[A-optimal experiment design.]
In this case $H$ is \emph{not} a PSD-DR function, so to obtain competitive
ratio bounds by our method we must construct a PSD-DR surrogate $H_S$
numerically.  Figure \ref{comrat_gamma_inv}, shows the competitive ratio bound
given by solving \eqref{cont-problem} for $h(u) = 1 - 1/(u+1)$ for different
values of the parameter $\gamma$. In Figure \ref{inv_cont_examples}, we have
provided examples of $h_S$ for different values of $\gamma$ and $u_{\max}$.
\end{example}

 \begin{figure}[h]  
   \begin{subfigure}{1\textwidth}
\centering    
                   \includegraphics[width=0.75\textwidth]{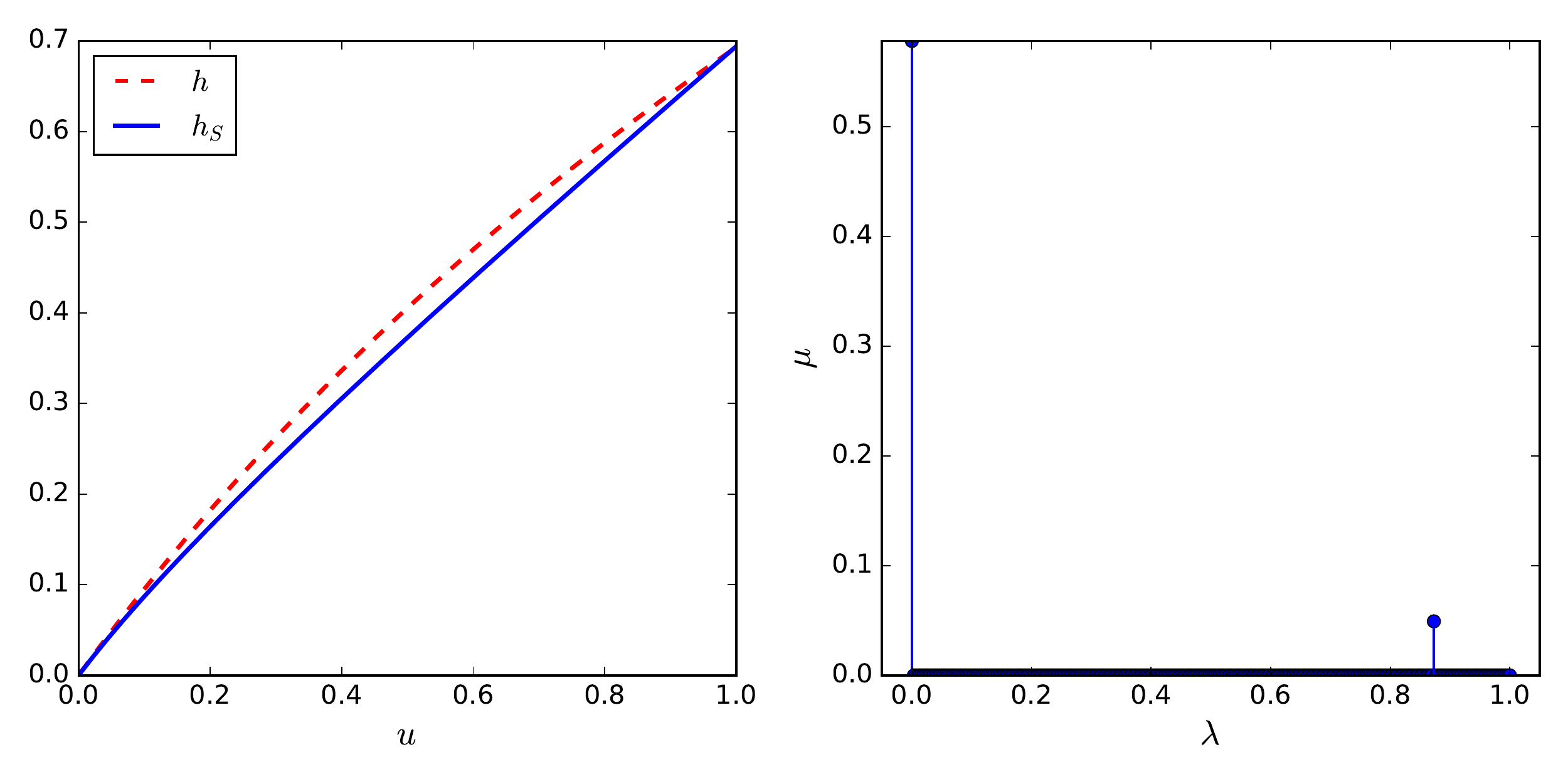}         
                      \caption{}
\label{log_cont_problem_11}
 \end{subfigure}
    \begin{subfigure}{1\textwidth}
\centering    
                   \includegraphics[width=0.75\textwidth]{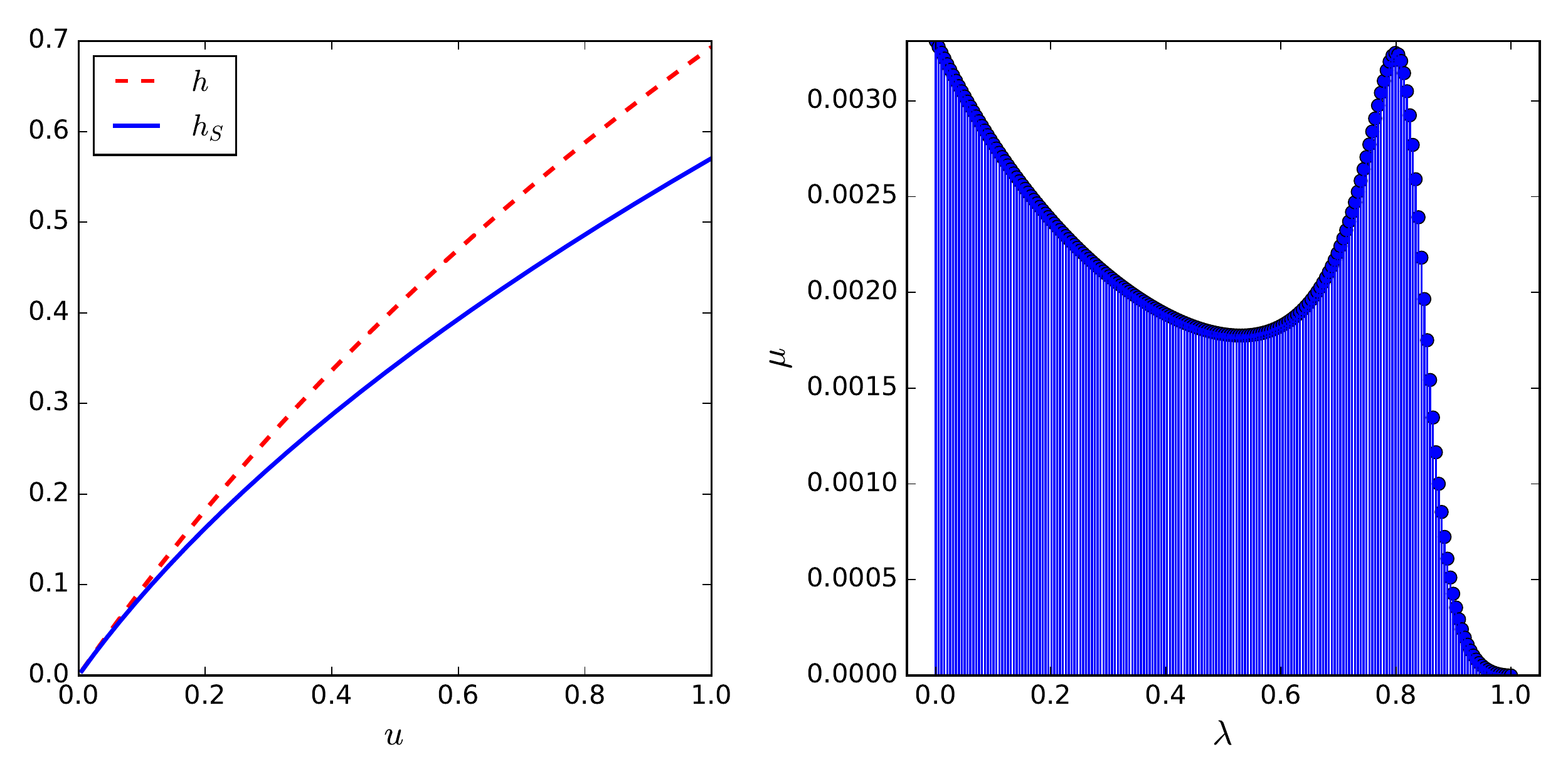}         
                      \caption{}
\label{log_cont_problem_41}
 \end{subfigure}
   \begin{subfigure}{1\textwidth}
\centering    
                   \includegraphics[width=0.75\textwidth]{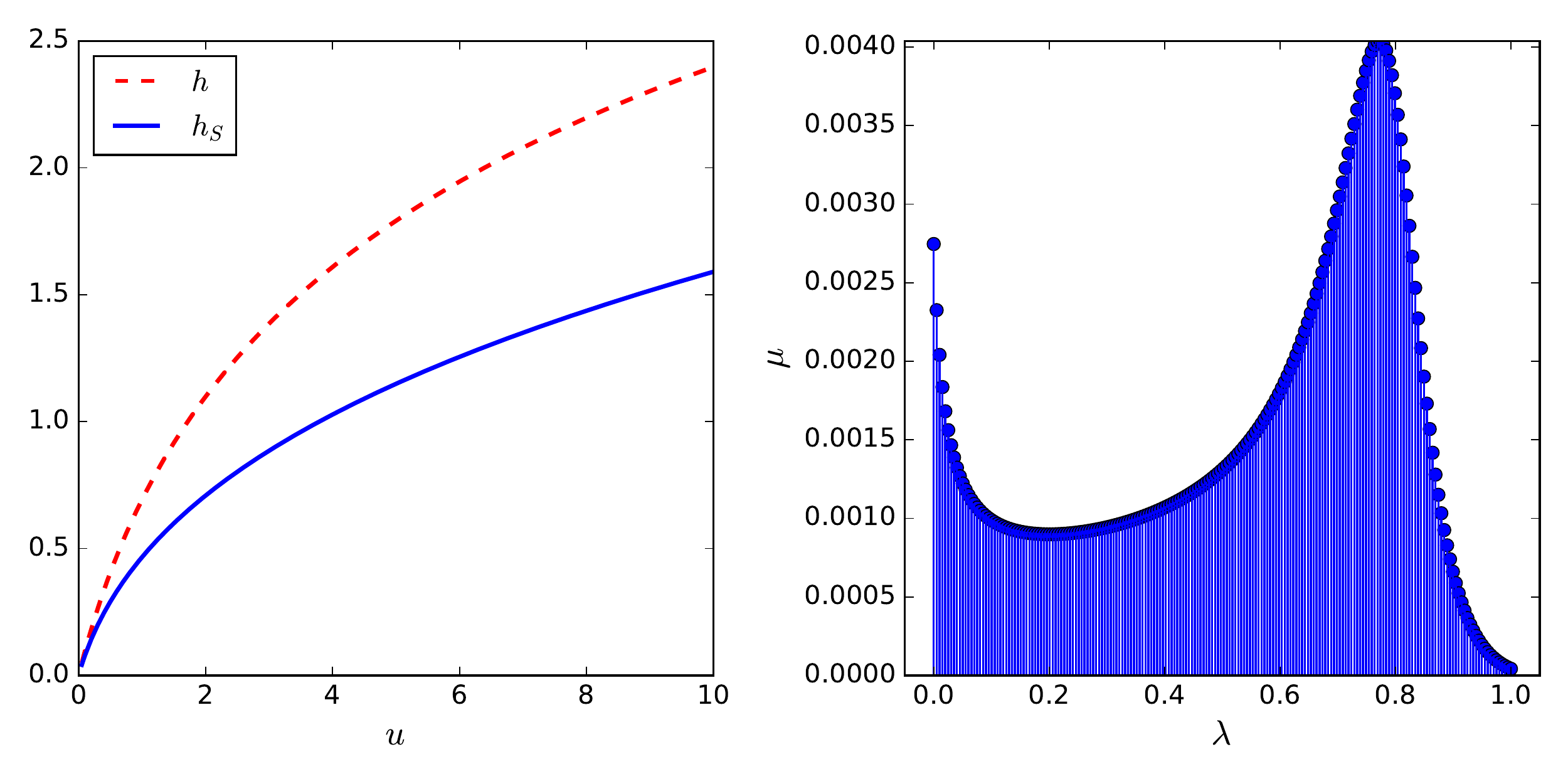}         
                      \caption{}
\label{log_cont_problem_410}
 \end{subfigure}
 \caption{The smoothed function $h_S$ and the corresponding measure $\mu$, when
$h(u) = \log(u+1)$. The smooth $h_S$ is found by solving the convex problem
\eqref{cont-problem} via discretization of measure $\mu$. In (a), $\gamma = 1$
and $u_{\max} = 1$. In (b), $\gamma = 4$ and $u_{\max} = 1$. In (c), $\gamma =
4$ and $u_{\max} = 10$.}
 \end{figure}

 \begin{figure}[h]  
   \begin{subfigure}{1\textwidth}
\centering    
                   \includegraphics[width=0.75\textwidth]{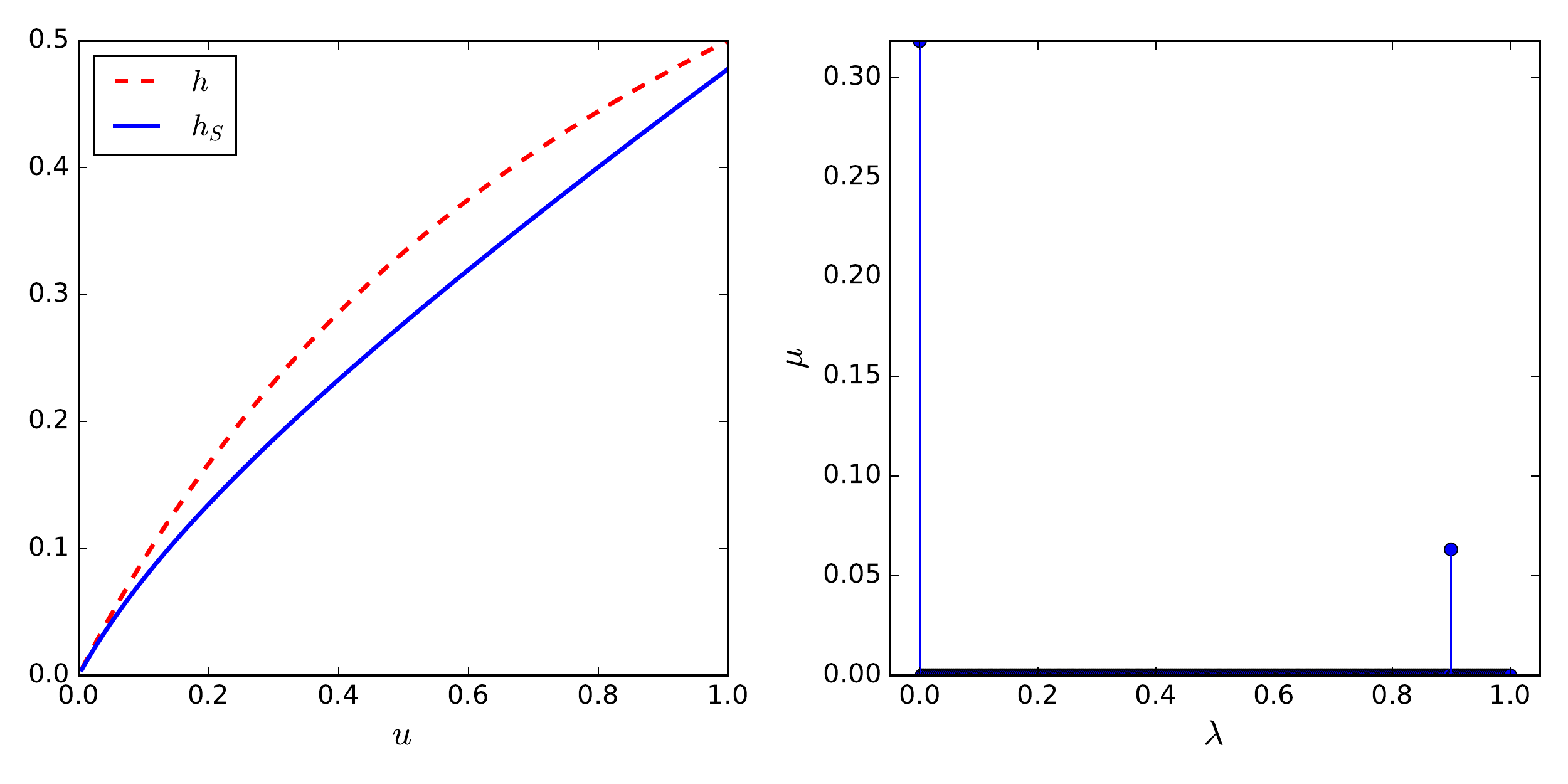}         
                      \caption{}
\label{inv_cont_problem_11}
 \end{subfigure}
    \begin{subfigure}{1\textwidth}
\centering    
                   \includegraphics[width=0.75\textwidth]{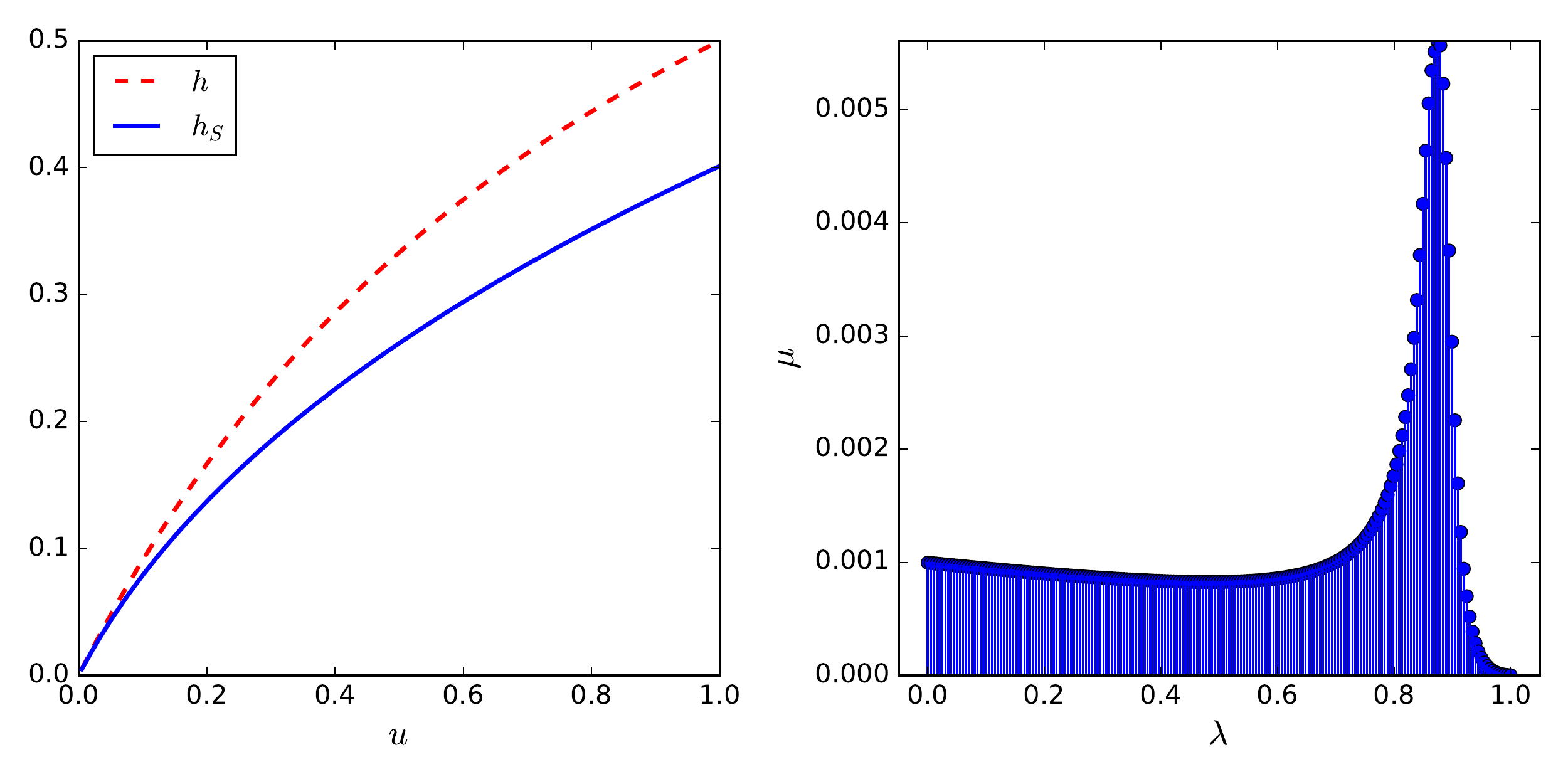}         
                      \caption{}
\label{inv_cont_problem_41}
 \end{subfigure}
   \begin{subfigure}{1\textwidth}
\centering    
                   \includegraphics[width=0.75\textwidth]{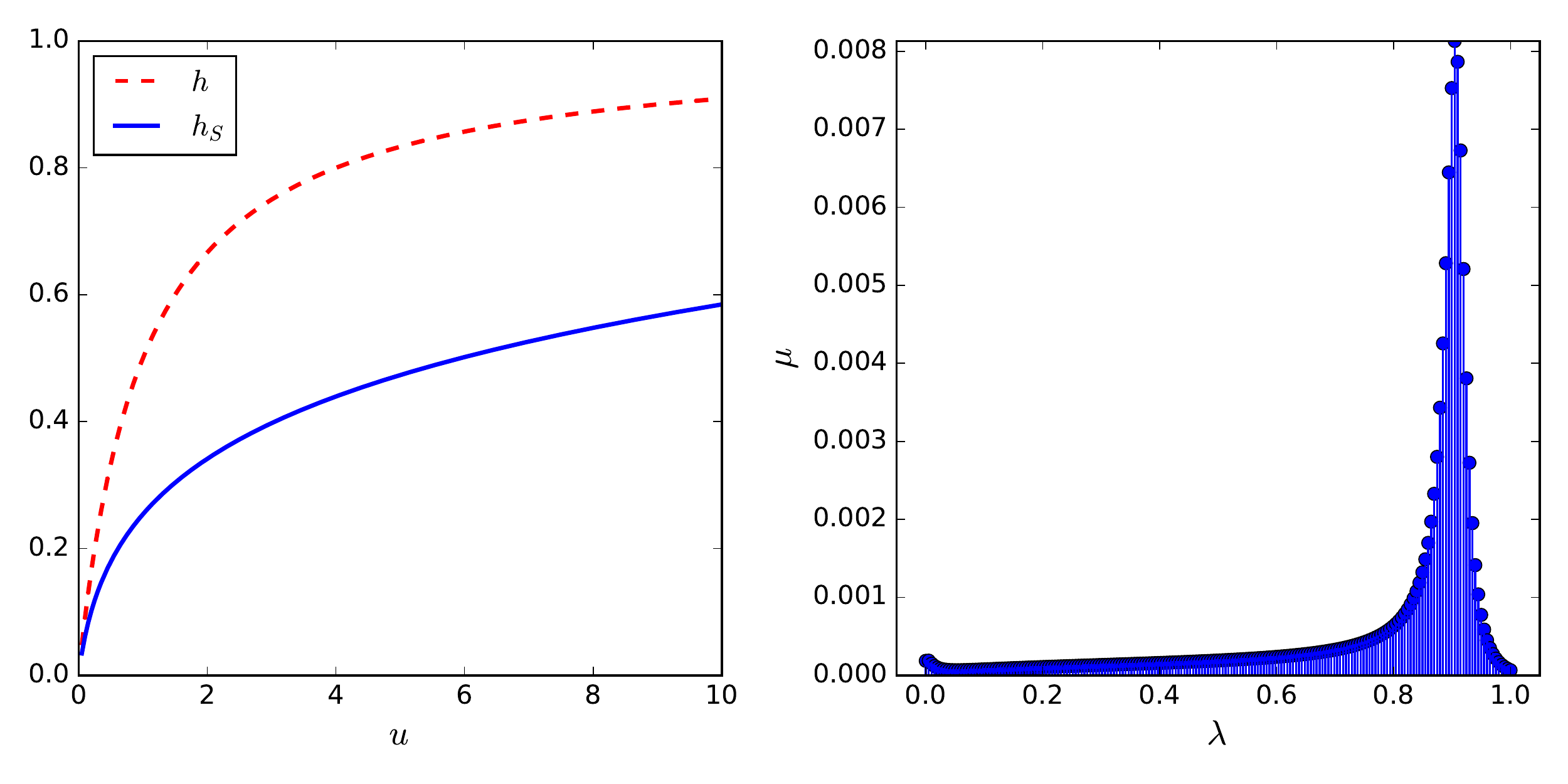}         
                      \caption{}
\label{inv_cont_problem_410}
 \end{subfigure}
 \caption{The smoothed function $h_S$ and the corresponding measure $\mu$, when
$h(u) = 1-\frac{1}{u+1}$. The smooth $h_S$ is found by solving the convex
problem \eqref{cont-problem} via discretization of measure $\mu$. In (a),
$\gamma = 1$ and $u_{\max} = 1$. In (b), $\gamma = 4$ and $u_{\max} = 1$. In
(c), $\gamma = 4$ and $u_{\max} = 10$.}
 \label{inv_cont_examples}
 \end{figure}

\clearpage

\section{Analysis of the sequential algorithm}
\label{sec:comp-ratio-sequential}
To design the functions $H_S$ and $G_S$ for the sequential algorithm (Algorithm~\ref{algorithm1}), 
we need additional information about the problem data. Since Algorithm~\ref{algorithm1} uses the dual 
variable from the \emph{previous} time step to assign the primal variable, the maximum length of each primal step, 
which is captured by the following two parameters, plays a role in the competitive ratio. 
Let  
\begin{align} \label{eq:rho-def2}
\rho_1 \geq \max_{t} c_t,  \quad\textup{and}\quad \rho_2 \geq  \max_{t} \lambda_{\max}\left(A_t\right).\end{align}
The next lemma is an analogue of Lemma~\ref{comp-ratio-pre} but for the sequential algorithm. 
\begin{lemma}
\label{comp-ratio-pre-seq2}
Let $G_S$ satisfy Assumption~\ref{assumption-gs}, $h_S$ satisfy Assumption~\ref{assumption-hs}, and 
suppose that the corresponding trace function $H_S$ satisfies the PSD-DR assumption. Suppose that 
	$\theta \leq c_t^{-1}\tr(A_t) \leq \Theta$ for all $t$, that $\rho_1$ and $\rho_2$ are defined as 
	in~\eqref{eq:rho-def2}, and 
	$u_{\max} \geq \lambda_{\max}\left(\sum_{t=1}^{m}A_t\hat{x}_t\right)$.~\footnote{Note 
	that we could choose, for instance, 
	$u_{\max} = b'\max_t c_t^{-1}\lambda_{\max}(A_t)$, but for certain classes of problems better bounds 
	may be available.} Then
	\begin{enumerate}
	\item If for a given $\gamma \geq 1$, $G_S$ satisfies:
	\begin{align}
	 &\!\!\!\!\gamma \left[G_S\left(u\right) -\rho_1 G_S'\left(u\right)\right] 
		\leq G^*(G_S'(u)) + \frac{\gamma}{e-1}h(\theta u)\quad \forall u\in [0,\infty)\label{eq:G_gamma_seq2}
	\end{align}
	and there exists $\beta > 0 $ such that $h_S$ satisfy the following inequlity:
	\begin{align}
	 &\!\!\!\!\gamma \left[h_S\left(u\right)+\rho_2 \left(h_S'\left(0\right)- h_S'\left(u\right)\right)\right] 
		\leq h^*(h_S'(u)) + \beta h(u)\quad \forall u\in [0,u_{\max}]\label{eq:h_beta_seq2}
	\end{align}
	then the iterates $\hat{x}_1,\ldots,\hat{x}_m$ of Algorithm~\ref{algorithm1} satisfy
	\begin{equation}
	\label{eq:crbound_seq2}
	 \textstyle{H\left(\sum_{t=1}^{m}A_t\hat{x}_t\right)} \geq \frac{1}{\gamma/(e-1)+\beta} D^\star.
	\end{equation}
	\item The iterates $\hat{x}_1,\ldots,\hat{x}_m$ of Algorithm~\ref{algorithm1} satisfy
	\begin{equation}
	\label{eq:bprime_seq2}
	 \sum_{t=1}^{m}c_t\hat{x}_t \leq b'\quad\textup{where}\quad b':= \rho_1 + \inf\{u: G_S'(u) \leq -h'(0)\Theta\}.
	\end{equation}
	\end{enumerate}
\end{lemma}

\begin{proof}
 	First we show that $\sum_{t=1}^{m}c_t\hat{x}_t \leq b'$. 
 	Arguing by contradiction, let $T$ be the smallest index such that 
 	$\sum_{t=1}^{T+1}c_t\hat{x}_t > b'$. Then, $\sum_{t=1}^{T}c_t\hat{x}_t > b'-\rho_1$ (since $c_{T+1}\hat{x}_{T+1} \leq \rho_1$).
	We have that 
	$\hat{z}_T = G_S'\left(\sum_{t=1}^{T}c_t\hat{x}_t\right) < -h'(0)\Theta$, 
	where the inequality holds by the definition of $b'$. By the PSD-DR assumption, 
 	$\nabla H_S\left(\sum_{t=1}^{T}A_t\hat{x}_t\right) \preceq \nabla H_S(0) = h_S'(0)I$. Then
 	\begin{equation*}
  	\inner{\hat{Y}_T}{A_{T+1}} + c_{T+1}\hat{z}_T  \leq c_{T+1}\left(h_S'(0)c_{T+1}^{-1}\tr(A_{T+1}) + \hat{z}_T\right) < 0.
	\end{equation*}
	It follows that $\hat{x}_{T+1}=0$, contradicting our choice of $T$. Hence $\sum_{t=1}^{m}c_t\hat{x}_t \leq b'$. 

 	We now bound on the competitive ratio. 
  	Let $U = \sum_{t=1}^{m} A_t \hat{x}_t$, $u = \sum_{t=1}^{m}c_t \hat{x}_t$, 
 	$\hat{Y}_m = \nabla H_S(U)$, and $\hat{z}_m = G_S'(u)$. 
	First we note that
	\begin{align}
	 \sum_{t=1}^{m}\left[\inner{A_t\hat{x}_t}{\hat{Y}_{t-1} - \hat{Y}_{t}} + c_t\hat{x}_t(\hat{z}_{t-1} - \hat{z}_{t})\right]
	&\leq \sum_{t=1}^{m}\left[\rho_2\tr(\hat{Y}_{t-1} - \hat{Y}_t) + \rho_1(\hat{z}_{t-1} - \hat{z}_t)\right]\nonumber\\
	& = \rho_2(\tr(\hat{Y}_0 - \hat{Y}_m)) - \rho_1(\hat{z}_m)\label{rho-bound2}
	\end{align}
	where the 
	inequality holds because $\hat{Y}_{t} \preceq \hat{Y}_{t-1}$ (by the PSD-DR assumption), 
	and, similarly, $\hat{z}_m \leq \cdots \leq \hat{z}_0 = 0$ (since $G_S'(0) = 0$). 
 	By the primal allocation rule in Algorithm \ref{algorithm1}, we have 
 		$\hat{x}_t \left( c_{t} \hat{z}_{t-1} + \inner{A_{t}}{\hat{Y}_{t-1}}\right) \geq 0$.
 	Combining this with the concavity of $H_S$ and $G_S$, we get
 	\begin{align*}
 		H_S\!\left(\textstyle{\sum_{s=1}^{t}} A_s\hat{x}_s\!\right) +
 		G_S\!\left(\textstyle{\sum_{s=1}^{t}}c_s \hat{x}_s\!\right) - 
 		H_S\!\left(\textstyle{\sum_{s=1}^{t-1}}A_s\hat{x}_s\!\right) - 
 		G_S\!\left(\textstyle{\sum_{s=1}^{t-1}}c_s \hat{x}_s\!\right) \\
 		 + \hat{x}_t\!\left(c_t (\hat{z}_{t-1} - \hat{z}_{t}) + \inner{A_t}{\hat{Y}_{t-1} - \hat{Y}_{t}}\right) \!\geq 0.
 	\end{align*}
	Taking the sum over $t$, telescoping, and using $h_S(0) = G_S(0) = 0$, gives
 	\begin{align}\label{increasing-seq2}
 	H_S\left(U\right)
	+ G_S\left(u\right)
	 + \sum_{t=1}^{m}\left[\inner{A_t\hat{x}_t}{\hat{Y}_{t-1} - \hat{Y}_t} + 
		c_t\hat{x}_t(\hat{z}_{t-1} - \hat{z}_t)\right]\geq 0.
 	\end{align} 

Now, the proof follows the same step as the proof of Lemma~\ref{comp-ratio-pre}
and uses \eqref{duality-gap-seq} and the above inequalities.
\begin{align*}
&H\left(U\right)  - {D_{\rm seq}}  \quad \quad \text{ } \\
&\geq H\left(U\right) + H^*(\hat{Y}_m) -H_S\left(U\right) - G_S\left(u\right) +  G^*\left(\hat{z}_m\right)   \\
& \quad +  \sum_{t=1}^{m} \hat{x}_t\left[\inner{ A_t }{\hat{Y}_{t} - \hat{Y}_{t-1}} + c_t\left(\hat{z}_{t} - \hat{z}_{t-1}\right)\right] \qquad \qquad\quad\qquad\qquad\qquad\qquad\;\; \text{ By \eqref{duality-gap-seq}} \\
&\geq \left[1 - \textstyle{\frac{\gamma}{e-1}}\right] H(U) +  H^*(\hat{Y}_m)-H_S\left(U\right) - \left(1-\gamma\right) G_S\left(u\right)   - \gamma \rho_1 G_S'\left(u\right)\\
 &\quad +  (\gamma + (1-\gamma))\sum_{t=1}^{m} \hat{x}_t\left[\inner{ A_t }{\hat{Y}_{t} - \hat{Y}_{t-1}} + c_t\left(\hat{z}_{t} - \hat{z}_{t-1}\right)\right] \qquad\quad\;\;  \textup{By \eqref{eq:G_gamma_seq2} and \eqref{eq:theta-ineq}} \\
&\geq \left[1 - \textstyle{\frac{\gamma}{e-1}}\right] H\left(U\right) +  H^*(\hat{Y}_m)-H_S\left(U\right) + \left(1-\gamma\right) H_S\left(U\right)  - \gamma \rho_1 \hat{z}_m\\
& \quad +  \gamma\sum_{t=1}^{m} \hat{x}_t\left[\inner{ A_t }{\hat{Y}_{t} - \hat{Y}_{t-1}} + c_t\left(\hat{z}_{t} - \hat{z}_{t-1}\right)\right] \qquad \qquad\qquad\qquad\qquad\qquad  \textup{ By \eqref{increasing-seq2}} \\
&\geq  \left[1 - \textstyle{\frac{\gamma}{e-1}}\right] H(U) +  H^*(\hat{Y}_m) - \gamma \left[H_S(U) + \rho_2 \tr\left(\hat{Y}_0 - \hat{Y}_m\right)\right] \qquad\quad\qquad \text{ By \eqref{rho-bound2}}\\
&\geq \left[1-\textstyle{\frac{\gamma}{e-1}} -\beta\right] H\left(U\right)\qquad\qquad\qquad\qquad\quad\;\;\;\;\,\qquad\qquad\qquad\qquad\qquad\qquad\textup{By \eqref{eq:h_beta_seq2}}  .
\end{align*}
Applying Lemma~\ref{DseqDstar} completes the proof.
\qed
\end{proof}

For the sequential algorithm to minimize the bound $b'$ on the budget consumption we choose $G_S$ such that 
\vspace{-3pt}
\begin{align}\label{G_S_seq}
G_S'\left(u\right) = - \frac{\theta\gamma}{(b+\rho_1 \gamma)(e-1)} 
	\int_{0}^{u} \exp\left(\frac{\gamma}{b + \rho_1 \gamma} (u-v)\right) h'\left(\theta v\right) \; d v.
\end{align}
To find $h_S$ for the sequential algorithm, the problem \eqref{cont-problem} is modified to:
\begin{align}\label{cont-problem-seq}
 &\mbox{minimize}_{\beta,y,\mu}{ \quad  \beta}\\ \notag
&\mbox{subject to} \; \gamma\!\int_{0}^{u}\! y(s) ds + \gamma \rho_2 (y(0)-y(u))  - h^*(y(u)) \leq \beta  h\left(u\right)\;\; \forall u \in [0,u_{\max}]\\ \notag
& \quad y(t) = \int_{0}^{1} \frac{1}{t \lambda + (1-\lambda)} \; d \mu(\lambda)\\
&\quad \textup{$\mu$ a positive measure supported on $[0,1]$.} \notag
\end{align}
Let $\beta(\gamma)$ be the optimal value of $\beta$ in~\eqref{cont-problem-seq}. 
Combining our arguments 
gives the following analogue of Theorem~\ref{comp-ratio} for the sequential algorithm. It describes the 
tradeoff between budget consumption and competitive ratio achieved for the sequential algorithm when $G_S$ and $H_S$
are designed optimally.
 \begin{thm}\label{comp-ratio-sequential}
	 If $G_S$ is defined by~\eqref{G_S_seq}, and $H_S$ is the trace function corresponding to $\int_{0}^uy(s)\;ds$, 
	where $y$ is optimal for~\eqref{cont-problem-seq}, then the iterates, $\hat{x}_1,\ldots,\hat{x}_m$ of Algorithm~\ref{algorithm1} satisfy
   	\begin{align*}
 	\sum_{t=1}^{m}c_t \hat{x}_t \leq G_S'^{-1}( -h'(0)\Theta)+\rho_1
\quad\textup{and}\quad
 		\textstyle{H\left( \sum_{t=1}^{m} A_t \hat{x}_t \right) \geq  \frac{1}{\gamma/ (e-1) + \beta(\gamma)} D^\star}
   	\end{align*}
where $c_t^{-1}\tr(A_t)\leq \Theta$ for all $t$.
\end{thm}
\begin{proof}
The proof is the natural analogue of Theorem~\ref{comp-ratio}.\qed
\end{proof}


\section{Numerical Experiments}
\label{sec:numerical-exp}
In this section, we provide the results of our numerical experiments and
discuss the numerical implementation of smoothing design.

In our numerical simulations, we consider the D-optimal ($H(U) = \log\det(I +
U)$) and the A-optimal ($H(U) = n - {\rm tr} ((I+U)^{-1})$) experiment design
problems. The $t^{th}$ matrix $A_t = a_t a_t^T$ is generated as follows: we
sample a vector $\eta$ uniformly at random from $\{-1,1\}^n$ and then set $a_t
= \frac{ \sqrt{m-t+1}} { \sqrt{n}} \eta$. This is an \emph{adversarial}
weighting, inspired by worst-case examples for online LP, e.g.,
\cite{buchbinder2009design}.

We compare the performance of Algorithm \ref{algorithm2} (sequential updates)
with two choices of $H_S$. In one case, we use the smoothed $H_S$ given by
solving \eqref{cont-problem}. In the second case, we use $H$ without smoothing
($H_S = H$). For both cases, we use the smooth $G_S$ given in \eqref{G_S}.
(Note that a pure greedy algorithm without smoothing $G$ simply picks the first
$b$ experiments and is a trivial algorithm, which we did not try.) We varied
the parameter $\gamma$ and plotted the competitive ratio versus the budget used
by the algorithm in Figure \ref{fig:simulation_result}. For each value of
$\gamma$, we have $10$ random repeats.

 \begin{figure}
 \begin{subfigure}{0.47\textwidth}
\centering    
                \includegraphics[width=.9\textwidth]{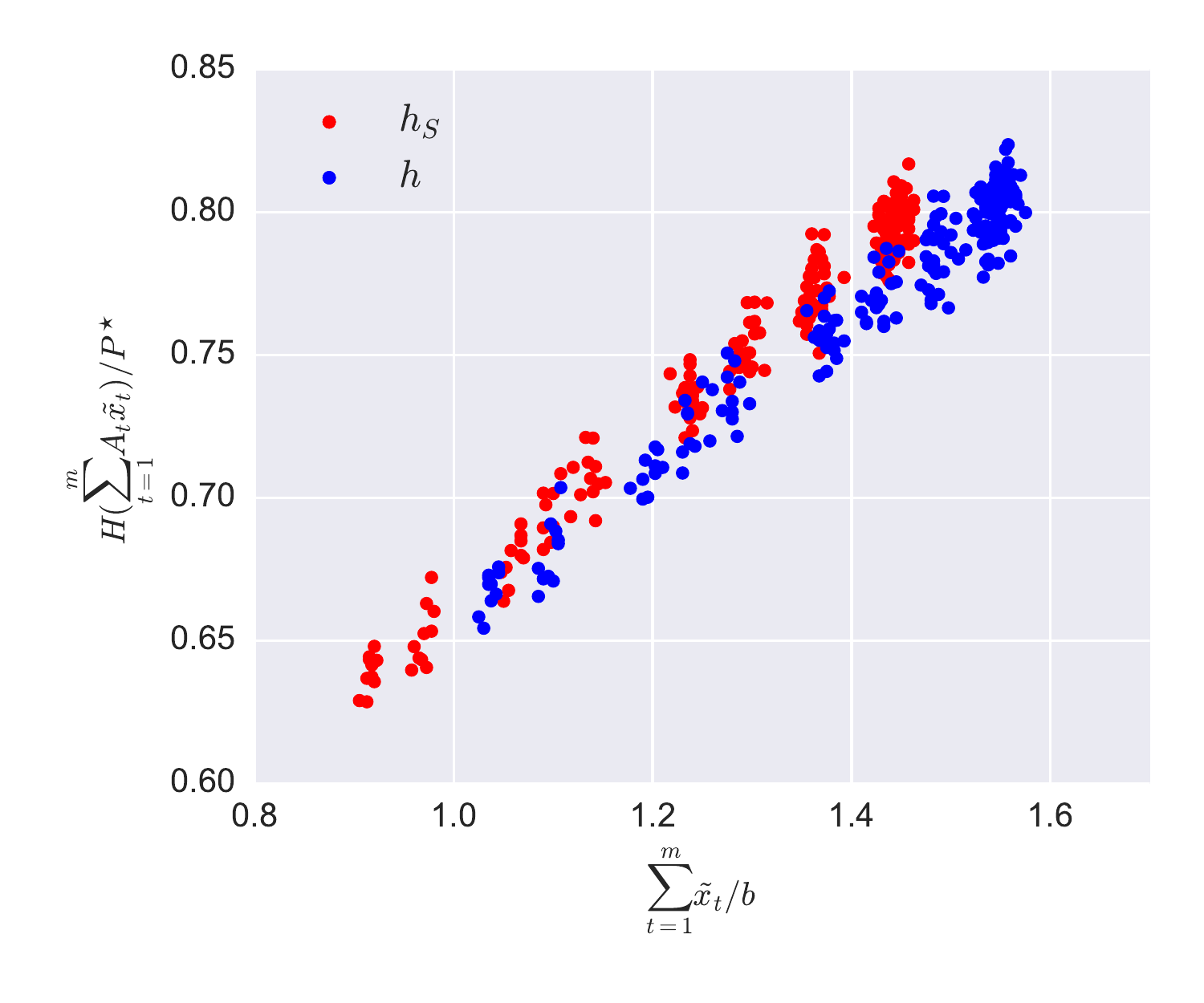}        
             \label{}   
 \end{subfigure}
  \begin{subfigure}{0.47\textwidth}
\centering    
                \includegraphics[width=.9\textwidth]{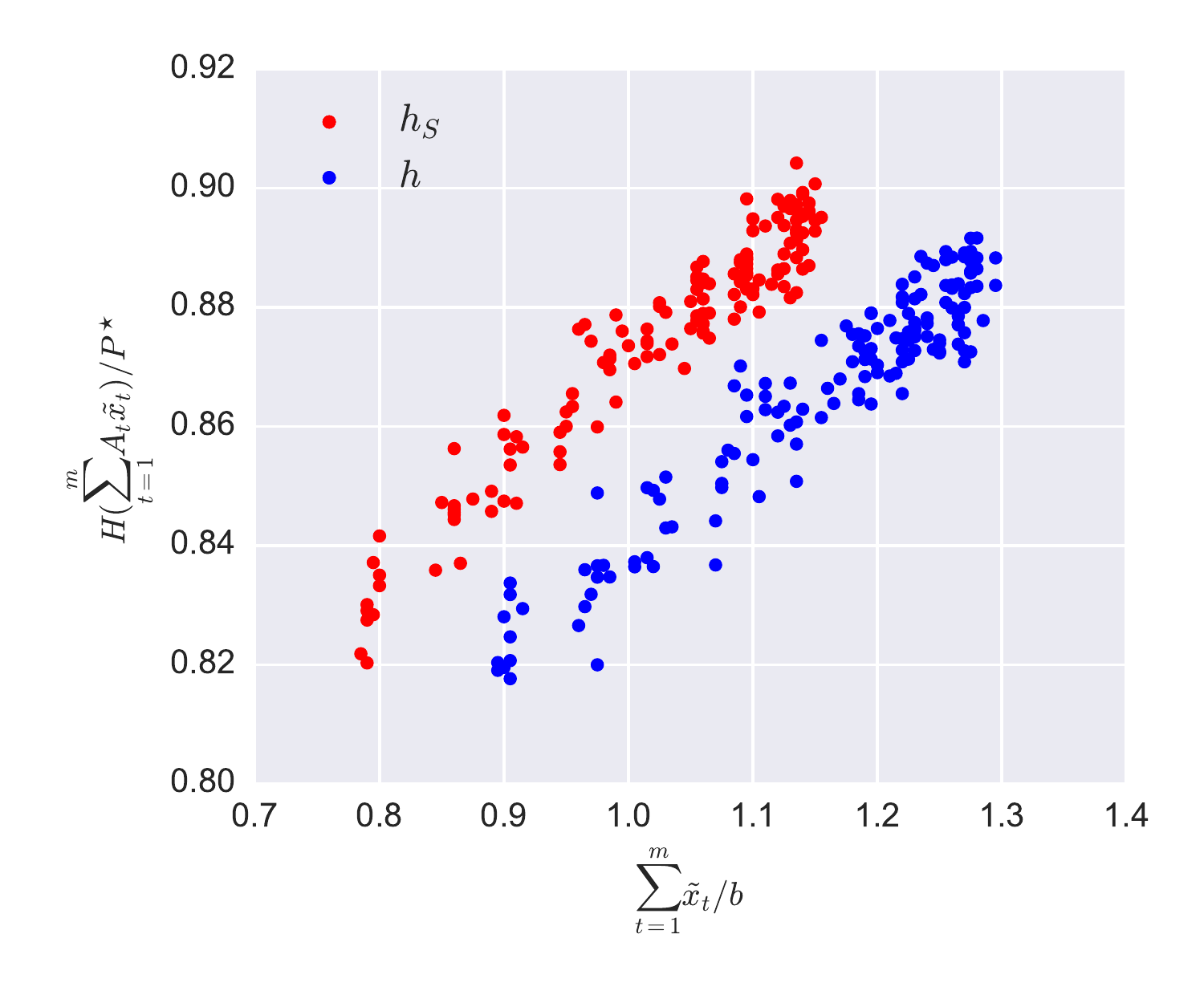}        
              \label{}     
 \end{subfigure}
\caption{Performance of Algorithm 2 (sequential updates) for online D-optimal
and A-optimal experiment design with two choices of $H_S$. (Left) D-optimal
experiment design. (Right) A-optimal experiment design. }
 \label{fig:simulation_result}\end{figure}

Next, we briefly discuss the numerical implementation of the smoothing design
problems for the simultaneous algorithm, introduced in
Section~\ref{sec:compratio}. Similar ideas apply for the smoothing design
problems for the sequential algorithm. We note that the algorithm does not
require $G_S$ itself, but only $G_S'(u)$. The formula for $G_S'$ can be
computed up to desired accuracy using Gauss-Leguerre
quadrature~\cite{reiz1943numerical}.

To solve problem \eqref{cont-problem}, we restrict $\mu$ to be an atomic
measure supported on the $q+1$ points $\lambda_j = j/q\in [0,1]$ for
$j=0,1,\ldots,q$.  The decision variables are then $\beta$ and
$\mu_j:=\mu(\lambda_j)$ for $j=0,1,\ldots,q$. Rather than imposing the
constraint for all $u\in [0,u_{\max}]$ we impose it on a non-uniformly sampled
subset. In particular, we sample $u$ more densely where $h$ has a larger local
Lipschitz constant by choosing the discretization points to be $u_i =
h^{-1}\left(i \; {u_{\rm max}}/{d}\right)$ for $i=0,1,\ldots,d$.
The optimal $y = h_S'$, which is all that is needed for the algorithm, is 
$y(u) = \sum_{j =0}^{q}  \frac{\mu_j}{u \lambda_j + \left(1-\lambda_j\right)}$.
Note that since the integral of every individual function in the summand 
representation of $h_S'$ satisfies PSD-DR, $h_S$ satisfies PSD-DR.

\section{Related Work and Discussion} 
\label{sec:related_work}

\paragraph{Submodularity and experiment design.}
Algorithms for the \emph{offline} optimal experiment design (both with and
without integer constraints) have been extensively studied
\cite{pukelsheim1993optimal}. 
Let $X_S$ denote a principal submatrix of $X$, then it is well known that the
set function $S \mapsto \log\det(X_S)$ is submodular. Based on this, greedy
subset selection is shown to have a $1-1/e$ approximation ratio
\cite{bouhtou2010submodularity,shamaiah2010greedy,seeger2009submodularity} for
the (integer) experiment selection problem.  The paper
\cite{bian2017guarantees} also gives an approximation ratio for the weakly
submodular A-optimal design problem; also see \cite{wang2016computationally}
and references therein for statistical bounds for solving the convex relaxation
followed by a greedy post-processing. In the existing literature, however,
algorithms are assumed to have access to all possible experiments at the start,
and therefore the result do not apply to the worst-case \emph{online} setting
that we consider. There is a connection between analysis of online algorithm
under stochastic i.i.d setting and greedy algorithm for submodular
maximization. We refer the reader to \cite{eghbali2017thesis} for the details
of this connection.

\paragraph{Online SDP problem of \cite{elad2016online}.} 
To the best of our knowledge, \cite{elad2016online} is the only existing work
that studies an online semidefinite program.  While the problem considered is
different from ours and its results do not apply to our setup, we briefly
discuss the idea.  \cite{elad2016online} considers a generalization of the
online \emph{covering} linear program to the semidefinite cone as follows,
\[ \textup{minimize}\;\; b^T y \quad\textup{subject to}\quad \begin{cases} \mathcal{A}(y) \succeq C & \\ y \geq 0,\end{cases}\]
where $\mathcal{A}$ denotes a linear map from vectors to matrices. The
algorithm receives a sequence of PSD matrices $C^{(0)} \preceq C^{(1)} \preceq
\ldots \preceq C^{(m)} = C$ over time, and needs to increase the variable $y$
so as to satisfy the new matrix covering constraint. The dual of this problem
is a packing problem, however the model for the online information the
algorithm receives is still the sequence of matrices $C^{(i)}$ which is
different from receiving a new experiment and bid price to decide on, as in our
model. Indeed, extending the \emph{packing} linear program to the PSD cone yields a different problem, i.e.,
\[ \textup{maximize}\;\;c^Tx \quad\textup{subject to}\quad \begin{cases}\mathcal{A}(x) \preceq B&\\ x \geq 0,&\end{cases}\]
which is in line with the setting we consider in this paper, after a
reformulation to bring the matrix constraint to the objective in a penalized
form (which corresponds to $H$). The possibility of using L\"{o}wner's
representation to design a matrix penalty function for this matrix-valued
budget is an interesting direction for future work.

\paragraph{Online Learning, Regret, and FTRL.} 
As mentioned earlier, the dual update in Algorithm \ref{algorithm1} is the same
as in Follow-the-Regularized-Leader (FTRL) algorithm with $-{H_S}^*$ as the
regularizer. This primal-dual perspective has been used in
\cite{shalev2007primal} for design and analysis of online learning algorithms.
In the online learning literature, the goal is to derive a bound on the
\emph{regret} that optimally depends on the horizon, $m$; whereas in this work
we study the competitive ratio for the algorithm that depends on the functions
$H$ and $G$.  In order to optimize the competitive ratio, the regularization
functions should be crafted based on $H$, and a general choice of
regularization which yields an optimal regret bound in terms of $m$ is not
enough for a competitive ratio argument, so existing results in online learning
do not address our aim. There are however some shared proof steps in the
analysis that are worth exploring further.

\subsection*{Acknowledgements}
The authors thank Omid Sadeghi-Meibodi for helpful comments. The work of MF and
RE was supported in part by grants ONR N000141612789, NSF CCF 1409836, NSF
Tripods 1740551, and ONR MURI N000141612710. Part of this work was done while
RE and MF were visiting the Simons Institute for the Theory of Computing,
partially supported by the DIMACS/Simons Collaboration on Bridging Continuous
and Discrete Optimization through NSF grant CCF-1740425.

\bibliographystyle{alpha}
\bibliography{OnRef}
\appendix 
\section{Additional proofs}
\label{app:Dgap}
Here, we provide additional proofs not given in detail in the body of the paper.

\begin{proof}[{Proof of Lemma~\ref{lem:duality-gap}}]
By the definition of $D_{\textup{sim}}$, the definition of $\tilde{x}_t$, and the concavity of $H_S$ and $G_S$, we have that 
\begin{align} \notag
	 D_{\textup{sim}} & =\textstyle{ \sum_{t=1}^{m}}\left[\inner{A_t\tilde{x}_t}{\tilde{Y}_t} + c_t\tilde{x}_t\tilde{z}_t\right] - H^*(\tilde{Y}_m) - G^*(\tilde{z}_m)\\ \notag
	& \leq\textstyle{ \sum_{t=1}^{m}}\textstyle{\left[H_S\left(\sum_{s=1}^{t}A_s\tilde{x}_s\right) - H_S\left(\sum_{s=1}^{t-1}A_s\tilde{x}_s\right)\right.}\\ \notag
	& \qquad\qquad \textstyle{\left. + G_S\left(\sum_{s=1}^{t}c_s\tilde{x}_s\right) - G_S\left(\sum_{s=1}^{t-1}c_s\tilde{x}_s\right)\right]} - H^*(\tilde{Y}_m) - G^*(\tilde{z}_m)\\ \notag
	& = \textstyle{H_S(\sum_{s=1}^{m}A_s\tilde{x}_s) + G_S(\sum_{s=1}^{m}c_s\tilde{x}_s)} - H^*(\tilde{Y}_m) - G^*(\tilde{z}_m).
\end{align}
The inequality follows from concavity of $G_S$ and $H_S$. The final equality holds by telescoping the sum and using the fact that $H_S(0) = 0 = G_S(0)$. 
For the sequential algorithm we can write:
\begin{align*}
	 D_{\textup{seq}} & = \textstyle{\sum_{t=1}^{m}}\left[\inner{A_t\hat{x}_t}{\hat{Y}_{t-1}} + c_t\hat{x}_t\hat{z}_{t-1}\right] - H^*(\hat{Y}_m) - G^*(\hat{z}_m)\\
	 & = \textstyle{\sum_{t=1}^{m}} \left[\inner{A_t\hat{x}_t}{\hat{Y}_{t}} + c_t\hat{x}_t\hat{z}_{t}\right]- H^*(\hat{Y}_m) - G^*(\hat{z}_m)  + 
	 \textstyle{\sum_{t=1}^{m}} \left[\inner{A_t\hat{x}_t}{\hat{Y}_{t-1} - \hat{Y}_{t}} + c_t\hat{x}_t\left(\hat{z}_{t-1} - \hat{z}_{t}\right)\right]
		\end{align*}
Now, the rest follows similar to steps as the simultaneous case.
\end{proof}

\begin{proof}[{Proof of Lemma~\ref{DseqDstar}}]
	We write out the argument for the inequality $D_{\textup{sim}} \geq
D^\star$. The argument showing that $D_{\textup{seq}} \geq D^\star$ is
identical. 
	We first show that the PSD-DR assumption on $H_S$ implies  
	\begin{equation}	\label{eq:Dstar-ineq}
		 \sum_{t=1}^{m}\left(\inner{A_t}{\tilde{Y}_t}+c_t\tilde{z}_t\right)_+ \geq \sum_{s=1}^{m}\left(\inner{A_s} {\tilde{Y}_m}+c_s\tilde{z}_m\right)_+.
	\end{equation}
	Since $A_s \in S_+^n$ and $\tilde{x}_s \geq 0$ for all $s\in [m]$, it follows that $\sum_{s=1}^{t}A_s \tilde{x}_s \preceq \sum_{s=1}^{m}A_s\tilde{x}_s$
	for all $t\in [m]$. Since $\tilde{Y}_t = \nabla H_S\left(\sum_{s=1}^{t}A_s\tilde{x}_s\right)$, 
	if $H_S$ satisfies the PSD-DR assumption then $\tilde{Y}_t \succeq \tilde{Y}_m$ for all $t\in [m]$.
	By a similar argument, since $G_S$ is concave, $\tilde{z}_t \geq \tilde{z}_m$ for all $t\in [m]$.  
	Since $A_t \in S_+^n$ and $c_t \geq 0$ for all $t\in [m]$,  
	\[ \inner{A_t}{\tilde{Y}_t} + c_t\tilde{z}_t \geq 
	\inner{A_t}{\tilde{Y}_m} + c_t\tilde{z}_t \geq \inner{A_t}{\tilde{Y}_m} + c_t\tilde{z}_m\]
	for all $t\in [m]$. Taking the positive part and then summing establishes~\eqref{eq:Dstar-ineq}. To conclude that 
	$D_{\textup{sim}}\geq D^\star$, we need only observe that $D^\star$ is a lower bound on the dual 
	objective~\eqref{eq:main-dual} evaluated at $(\tilde{Y}_m,\tilde{z}_m)$.  
\end{proof}

\end{document}